\documentclass[11pt,a4paper]{article}

\usepackage[left=1.5cm,right=3cm,top=2cm,bottom=2cm]{geometry}
\usepackage{array}
\newcolumntype{x}{>{\vbox to 5ex\bgroup\vfill\centering\arraybackslash\hspace{0pt}}m{2.2cm}<{\egroup}}
\usepackage[fleqn]{amsmath}
\usepackage{amssymb,latexsym}
\usepackage{amsthm}
\usepackage[colorlinks=true,linkcolor=black,citecolor=black,urlcolor=black]{hyperref}
\usepackage{enumitem}
\usepackage{graphicx}
\usepackage{tikz}

\newtheorem{theorem}{Theorem}[section]
\newtheorem{corollary}[theorem]{Corollary}
\newtheorem{lemma}[theorem]{Lemma}

\newcommand{\Z}{\mathbb{Z}}

\newenvironment{config}{\tt\small\setlength{\parskip}{1pt}}{}

\DeclareMathOperator{\Aut}{Aut}

\DeclareMathOperator{\per}{per}
\title{Colouring problems for symmetric configurations with block size 3}

\author{
Grahame Erskine\thanks{Open University, Milton Keynes, UK}\footnotemark[1]\\ \texttt{\small grahame.erskine@open.ac.uk}\\ \small{(corresponding author)}
\and Terry Griggs\footnotemark[1]\\ \texttt{\small terry.griggs@open.ac.uk}
\and Jozef \v{S}ir\'a\v{n}\thanks{Open University, Milton Keynes, UK and Slovak University of Technology, Bratislava, Slovakia}\footnotemark[2]\\ \texttt{\small jozef.siran@open.ac.uk}
}

\date{}

\begin{document}
\maketitle
\let\thefootnote\relax
\let\thefootnote\relax

\vspace*{-8ex}
\begin{abstract}
\noindent The study of symmetric configurations $v_3$ with block size 3 has a long and rich history. In this paper we consider two colouring problems which arise naturally in the study of these structures. The first of these is weak colouring, in which no block is monochromatic; the second is strong colouring, in which every block is multichromatic. The former has been studied before in relation to blocking sets. Results are proved on the possible sizes of blocking sets and we begin the investigation of strong colourings.  We also show that the known $21_3$ and $22_3$ configurations without a blocking set are unique and make a complete enumeration of all non-isomorphic $20_3$ configurations. We discuss the concept of connectivity in relation to symmetric configurations and complete the determination of the spectrum of 2-connected symmetric configurations without a blocking set. A number of open problems are presented.
\end{abstract}

\section{Introduction}\label{sec:intro}
In this paper we will be concerned with symmetric configurations with block size 3 and, more particularly, two colouring problems which arise naturally from their study. First we recall the definitions. A \emph{configuration} $(v_r,b_k)$ is a finite incidence structure with $v$ points and $b$ blocks, with the property that there exist positive integers $k$ and $r$ such that
\begin{enumerate}[label=(\roman*),topsep=0pt,itemsep=0pt]
	\item each block contains exactly $k$ points;
	\item each point is contained in exactly $r$ blocks; and
	\item any pair of distinct points is contained in at most one block.
\end{enumerate}
A configuration is said to be \emph{decomposable} or \emph{disconnected} if it is the union of two configurations on distinct point sets. We are primarily interested in indecomposable (connected) configurations, and so unless otherwise noted, this is assumed throughout the paper.

If $v=b$ (and hence necessarily $r=k$), the configuration is called \emph{symmetric} and is usually denoted by $v_k$. We are interested in the case where $k=3$. Such configurations include a number of well-known mathematical structures. The unique $7_3$ configuration is the Fano plane, the unique $8_3$ configuration is the affine plane $AG(2,3)$ with any point and all the blocks containing it deleted, the Pappus configuration is one of three $9_3$ configurations and the Desargues configuration is one of ten $10_3$ configurations. Symmetric configurations have a long and rich history. It was Kantor in 1881~\cite{Kantor1881} who first enumerated the $9_3$ and $10_3$ configurations and in 1887, Martinetti~\cite{martinetti1887sulle} showed that there are exactly 31 configurations $11_3$.

It is natural to associate two graphs with a symmetric configuration $v_3$. The first is the \emph{Levi graph} or \emph{point-block incidence graph}, obtained by considering the $v$ points and $v$ blocks of a configuration as vertices, and including an edge from a point to every block containing it. It follows that the Levi graph is a cubic (3-regular) bipartite graph of girth at least six. The second graph is the \emph{associated graph}, obtained by considering only the points as vertices and joining two points by an edge if and only if they appear together in some block. Thus the associated graph is regular of valency 6 and order $v$. 

We note that symmetric configurations with block size 3 have also been studied in the context of 3-regular, 3-uniform hypergraphs. In this scenario the points of the configuration are identified with the vertices in the hypergraph and the blocks with the hyperedges; the condition that no pair of distinct vertices should be in more than one hyperedge is usually referred to as a \emph{linearity} condition in hypergraph terminology.

By a \emph{colouring} of a symmetric configuration $v_3$, we mean a mapping from the set of points to a set of colours. In such a mapping, if no block is monochromatic we have a \emph{weak colouring} and if every block is \emph{multichromatic} or \emph{rainbow} we have a \emph{strong colouring}. The minimum number of colours required to obtain a weak (resp. strong) colouring will be called the \emph{weak} (resp. \emph{strong}) \emph{chromatic number} and denoted by $\chi_w$ (resp. $\chi_s$). It is immediate from the definition that the strong chromatic number $\chi_s$ of a configuration is equal to the chromatic number of its associated graph.

Weak colourings have been studied before in relation to so-called blocking sets and in Section~\ref{sec:sizes} we begin the study of the sizes of these. In Section~\ref{sec:bsfree} we bring together various results concerning symmetric configurations without a blocking set which appear throughout the literature, some of which do not seem to be readily available. Section~\ref{sec:conn} is concerned with the connectivity of configurations and we complete the spectrum of 2-connected symmetric configurations without a blocking set. Our results on enumeration appear in Section~\ref{sec:enum}. In particular we extend known results by enumerating all symmetric configurations $20_3$ together with their properties, and prove that the known $21_3$ and $22_3$ configurations without a blocking set are the unique configurations of those orders with that property. Section~\ref{sec:strong} is concerned with strong colourings. To the best of our knowledge, both this topic and the sizes of blocking sets in Section~\ref{sec:sizes} appear to have been neglected and the results are new. Finally in Section~\ref{sec:open} we bring together some of the open problems raised by the work in this paper.

\section{Weak colourings}\label{sec:weak}
We begin with the following result which is a special case of Theorem~8 of~\cite{Bollobas1985}.
\begin{theorem}[Bollob\'as and Harris]
For every symmetric configuration $v_3$, either $\chi_w=2$ or $\chi_w=3$.
\end{theorem}
A \emph{blocking set} in a symmetric configuration is a subset of the set of points which has the property that every block contains both a point of the blocking set and a point of its complement. From this definition it is immediate that the complement of a blocking set is also a blocking set, and that the existence of a blocking set in a configuration is equivalent to  $\chi_w=2$. Empirical evidence indicates that almost all symmetric configurations $v_3$ contain a blocking set. Indeed, Table~\ref{tab:enum} shows that of the 122,239,000,083 connected configurations with $v\leq 20$, only 6 fail to have a blocking set.

For any $v\geq 8$, a configuration with a blocking set is very easy to construct. For $v$ even, the set of blocks generated by the block $\{0,1,3\}$ under the mapping $i\mapsto i+1\!\pmod{v}$ has a blocking set consisting of all the odd numbers (and hence, another consisting of all the even numbers). For $v$ odd and $v\geq 11$, construct the symmetric configuration $(v-1)_3$ as above and replace the blocks $\{0,1,3\}$ and $\{4,5,7\}$ with the blocks $\{1,3,c\}$, $\{4,7,c\}$ and $\{0,5,c\}$, where $c$ is a new point. The set of odd numbers is still a blocking set. The Fano plane does not have a blocking set, but all three $9_3$ configurations do.

The above extension operation can be summarised and generalised as follows.
\begin{itemize}[topsep=0pt,itemsep=0pt]
	\item Choose two non-intersecting blocks $\{a_1,a_2,a_3\}$ and $\{b_1,b_2,b_3\}$ such that the points $a_1$ and $b_1$ are not in a common block.
	\item Remove these blocks, introduce a new point $c$ and add three new blocks $\{c,a_2,a_3\}$, $\{c,b_2,b_3\}$ and $\{c,a_1,b_1\}$.
\end{itemize} 
This construction goes back to Martinetti~\cite{martinetti1887sulle}; see also~\cite{Boben2007}.

\subsection{Sizes of blocking sets}\label{sec:sizes}
Perhaps surprisingly, the cardinalities of blocking sets of those configurations $v_3$ for which $\chi_w=2$ do not seem to have been studied. Let $Q$ be such a blocking set and let $q=|Q|$. It is immediate that $\lceil v/3\rceil\leq q\leq\lfloor 2v/3\rfloor$. We have the following result.
\begin{theorem}\label{thm:configsize}
Let $v_3$ be a symmetric configuration with a blocking set $Q$ of cardinality $q$ where $\lceil v/3\rceil\leq q < \lfloor v/2\rfloor$. Then $v_3$ also has a blocking set $\overline{Q}$ of cardinality $q+1$.
\end{theorem}
\begin{proof}
Let $A$, of cardinality $\alpha$, be the set of blocks of the configuration which $Q$ intersects in one point and $B$, of cardinality $\beta$, be the set of blocks which $Q$ intersects in two points. Then $\alpha+\beta=v$ and $\alpha+2\beta=3q$. Thus since $q<\lfloor v/2\rfloor$, $\beta=3q-v<v-q$. Each block in $B$ contains two points in $Q$ and one point not in $Q$. Hence there exists a point $x\notin Q$ which is in no block of $B$, and so must be contained in three blocks of $A$. The set $\overline{Q}=Q\cup\{x\}$ is also a blocking set.
\end{proof}
Bearing in mind that if $Q$ is a blocking set for a configuration $v_3$ then so is $V\setminus Q$, it follows that the range of cardinalities of blocking sets of a configuration is continuous, and that configurations can be categorised by the minimum cardinality of a blocking set; a blocking set of this minimum cardinality will be called a \emph{minimal blocking set}. Configurations which have a blocking set of cardinality $q$ for all $q:\lceil v/3\rceil\leq q\leq\lfloor 2v/3\rfloor$ are relatively easy to construct.
\begin{theorem}\label{thm:minblocking}
Let $v\geq 9$. Then there exists a configuration $v_3$ having a blocking set of cardinality $q$ for all $q:\lceil v/3\rceil\leq q\leq\lfloor 2v/3\rfloor$.
\end{theorem}
\begin{proof}
In view of Theorem~\ref{thm:configsize}, it is sufficient to construct a symmetric configuration with a blocking set of cardinality $\lceil v/3\rceil$. Suppose first that $v\equiv 0\!\pmod 3, v\geq 9$. Let $v=3s$. Let the points of the configuration be $V=\{a_i,b_i,c_i:0\leq i\leq s-1\}$. Let the blocks be the sets $\{a_i,b_i,c_{i+1}\}$, $\{a_i,b_{i+1},c_i\}$, $\{a_{i+1},b_i,c_i\}$, $0\leq i\leq s-1$, subscript arithmetic modulo $s$. The set $Q=\{a_i:0\leq i\leq s-1\}$ is a blocking set.

Now suppose that $v\equiv 1\!\pmod 3,v\geq 10$. Construct a configuration $(v-1)_3$ as above. Introduce a new point $\infty_0$ and use Martinetti's extension operation, replacing the blocks $\{a_0,b_0,c_1\}$ and $\{a_1,b_1,c_2\}$ by blocks $\{\infty_0,b_0,b_1\}$, $\{\infty_0,a_0,c_1\}$ and $\{\infty_0,a_1,c_2\}$. The set $Q\cup \{b_1\}$ is a blocking set.

Finally, suppose that $v\equiv 2\!\pmod 3,v\geq 11$. Construct a configuration $(v-1)_3$ as above. Introduce a further new point $\infty_1$ and again use the extension operation, replacing the blocks $\{a_0,b_1,c_0\}$ and $\{a_1,b_2,c_1\}$ by blocks $\{\infty_1,b_1,b_2\}$, $\{\infty_1,a_0,c_0\}$ and $\{\infty_1,a_1,c_1\}$. The set $Q\cup\{b_1\}$ is again a blocking set.
\end{proof}
We note that the condition $v\geq 9$ in the above theorem is necessary; the unique $7_3$ configuration has no blocking set at all, and the unique $8_3$ configuration has a minimal blocking set of cardinality 4. Since Theorem~\ref{thm:minblocking} shows that a configuration $v_3$ with a minimal blocking set as small as possible exists for all $v\geq 9$, it is natural to ask what the range of possible sizes of minimal blocking sets might be. At the minimum end of the range, we are able to prove the following results.

\begin{theorem}\label{thm:nearmin}
	~
	
\begin{enumerate}[nosep,label=(\alph*)]
	\item There exists a configuration $v_3$ with a minimal blocking set of size $\lceil\frac{v}{3}\rceil+1$ for all $v\geq 8$.
	\item There exists a configuration $v_3$ with a minimal blocking set of size $\lceil\frac{v}{3}\rceil+2$ for $v=12$ and all $v\geq 15$.
\end{enumerate}
\end{theorem}
\begin{proof}
We deal first with part (a). First observe that from Table~\ref{tab:bssizes}, there exists such a configuration $v_3$ for $8\leq v\leq 16$. Let $\mathcal{A}$ be the set of all blocks of the configuration $8_3$ as given in the Appendix, i.e. 012, 034, 056, 135, 147, 246, 257, 367. This has a minimal blocking set of size 4. Let $\mathcal{B}$ be the set of blocks of a configuration $(3s)_3$ as given in Theorem~\ref{thm:minblocking}; i.e. the points are the set $V=\{a_i,b_i,c_i:0\leq i\leq s-1\}$ and the blocks are the sets $\{a_i,b_i,c_{i+1}\}$, $\{a_i,b_{i+1},c_i\}$, $\{a_{i+1},b_i,c_i\}$, $0\leq i\leq s-1$, subscript arithmetic modulo $s$. Replace the block $\{0,1,2\}$ by $\{a_0,1,2\}$ and the block $\{a_0,b_0,c_1\}$ by $\{0,b_0,c_1\}$ to form sets $\overline{\mathcal{A}}$ and $\overline{\mathcal{B}}$ respectively. The set $\overline{\mathcal{A}}\cup\overline{\mathcal{B}}$ is a connected configuration $(3s+8)_3$. We need to show that this has a minimal blocking set of size 4.

Considering the set $\overline{\mathcal{A}}$, a blocking set $Q$ must contain at least 4 points of the set $\{a_0,i:0\leq i\leq 7\}$ and further, in the special case that both $a_0,0\in Q$ it must contain at least 5 points. Otherwise, then by replacing the point $a_0$ by the point 0 to return to the set $\mathcal{A}$, the configuration $8_3$ would have a blocking set of size 3. Now consider the set $\overline{\mathcal{B}}$. In the above special case, $Q$ must contain at least $s-1$ elements of the set $V\setminus\{a_0\}$ and in all other cases, at least $s$ elements. So $Q$ has at least $s+4$ elements; to show that a minimal blocking set has exactly $s+4$ elements we may take a blocking set $Q=\{1,4,5,6,b_i:0\leq i\leq s-1\}$.

We next deal with configurations $(3s+9)_3$, $s\geq 3$. The procedure is precisely the same as the above case, except that we use the configuration $9_3$ as given in the Appendix, i.e. 012, 034, 056, 135, 147, 248, 267, 368, 578 which also has a minimal blocking set of size 4. In this case we take a blocking set $Q=\{1,4,5,6,b_i:0\leq i\leq s-1\}$.

Finally for configurations $(3s+10)_3$ we use one of the two configurations $10_3$ as given in the Appendix with a minimal blocking set of size 5, namely 012, 034, 056, 135, 178, 247, 268, 379, 469, 589. Again the procedure is as in the above two cases and we can take a blocking set $Q=\{1,4,5,6,7,b_i:0\leq i\leq s-1\}$.

Now we deal with part (b). From Table~\ref{tab:bssizes}, there exists such a configuration for $v\in\{12,15,16,17\}$. The configuration on the set $\mathbb{Z}_v$ generated by the block $\{0,1,3\}$ under the mapping $i\mapsto i+1\!\pmod v$ has a blocking set of size $\lceil\frac{v}{3}\rceil+2$ for $v\in\{18,21,22,23,24\}$; see Theorem~\ref{thm:blocking013}. The case where $v=19$ is of particular interest since $\lceil\frac{v}{3}\rceil + 2 = \lfloor\frac{v}{2}\rfloor$, the maximum size of a minimal blocking set.  Of the 7,597,039,898 connected configurations $19_3$, see~\cite{EGS2019} and Table~\ref{tab:enum}, only seven have a minimal blocking set of size 9 and these are given below.

\begin{config}
\-\ 012 034 056 137 189 25a 28b 3cd 46e 4cf 5gh 6gi 79h 7dg 8ei 9ef abc afh bdi

\-\ 012 034 056 137 145 236 258 469 7ab 7cd 8ae 8cf 9ag 9ch bdi beg dfh efi ghi

\-\ 012 034 056 137 158 239 2ab 46c 47d 5ef 6eg 78h 8fi 9af 9ei acd bcg bdh ghi

\-\ 012 034 056 137 158 239 2ab 47c 4de 5fg 68f 6hi 7dh 8ei 9ag 9bi acf beh cdg

\-\ 012 034 056 137 148 239 24a 5bc 5de 6bd 6cf 78g 79h 8ah 9ag bef cei dfi ghi

\-\ 012 034 056 137 158 239 2ab 457 46c 6de 78f 8gh 9ai 9bg acd bdh cei efh fgi

\-\ 012 034 056 137 145 236 257 468 79a 89b 8ac 9de adf bcg beh cfi dhi egi fgh
\end{config}

The most interesting of these is possibly the first one which is point-transitive; one of only three such configurations $19_3$, again see~\cite{EGS2019} and Table~\ref{tab:enum}.  Its Levi graph is arc-regular and has automorphism group of order 114.  It is the unique symmetric graph of order 38 and is graph F038A in the Foster census~\cite{foster88}. The configuration is cyclic and is isomorphic to the configuration generated by the block $\{0,1,8\}$ under the mapping $i\mapsto i+1\!\pmod{19}$. An example of a symmetric configuration on 20 points having a minimal blocking set of size 9 is as follows. 

\begin{config}
\-\ 012 034 056 135 146 237 245 678 79a 8bc 8de 9bf 9dg abh adi cej cfh egi fgj hij
\end{config}

So we may assume that $v\geq 25$. We follow closely the argument above. Let $\mathcal{A}$ and $\mathcal{A}'$ be the sets of all blocks of the configuration $8_3$ as given in the Appendix on point sets $\{0,1,\ldots,7\}$ and $\{0',1',\ldots,7'\}$ respectively. Let $\mathcal{B}$ be as in part (a). Replace the block $\{0,1,2\}$ by $\{a_0,1,2\}$, the block $\{0',1',2'\}$ by $\{b_0,1',2'\}$ and the blocks $\{a_0,b_0,c_1\}$ and $\{a_1,b_0,c_0\}$ by $\{0,b_0,c_1\}$ and $\{a_1,0',c_0\}$ to form sets $\overline{\mathcal{A}}$, $\overline{\mathcal{A}'}$ and $\overline{\mathcal{B}}$ respectively.

As in part (a), by considering the sets $\overline{\mathcal{A}}$ and $\overline{\mathcal{A}'}$, a blocking set $Q$ must contain at least 4 points of each of the sets $\{a_0,i:0\leq i\leq 7\}$ and $\{b_0,i':0\leq i\leq 7\}$; and if both $a_0,0\in Q$ at least 5 points of the former set and if both $b_0,0'\in Q$ at least 5 points of the latter set. Now by considering the set $\overline{\mathcal{B}}$, $Q$ must contain at least $s-2$ elements of the set $V\setminus\{a_0,b_0\}$ if $\{a_0,0,b_0,0'\}\subset Q$; $s-1$ elements if either $a_0,0\in Q$ or $b_0,0'\in Q$ but not both; and $s$ elements otherwise. In other words, $Q$ must contain at least $s+8$ elements in total. To show that a minimal blocking set has exactly $s+8$ elements, take $Q=\{1,4,5,6,1',4',5',6',c_i:0\leq i\leq s-1\}$. This deals with symmetric configurations $(3s+16)_3$, $s\geq 3$.

To deal with symmetric configurations $(3s+17)_3$, $s\geq 3$, the procedure is precisely the same except that for the set $\mathcal{A}$ we use the configuration $9_3$ as given in the Appendix. A minimal blocking set of size $s+8$ is again $Q=\{1,4,5,6,1',4',5',6',c_i:0\leq i\leq s-1\}$. Finally for configurations $(3s+18)_3$, $s\geq 3$, we also replace the set $\mathcal{A}'$ with the configuration $9_3$ on point set $\{0',1',\ldots,8'\}$. Again a minimal blocking set of size $s+8$ is $Q=\{1,4,5,6,1',4',5',6',c_i:0\leq i\leq s-1\}$.
\end{proof}

At the maximum end of the range, the situation appears to be much more difficult. Table~\ref{tab:bssizes} shows minimal blocking set sizes for symmetric configurations with $v\leq 17$. Of the $42,178,413$ such configurations, only $60$ have a minimal blocking set of size $\lfloor\frac{v}{2}\rfloor$; the 27 examples for $v\leq 14$ are shown in the Appendix. At $v=16$, it is noteworthy that only two of the very large number of configurations fail to have a blocking set of size 7; these are shown below.

\begin{config}
\-\ 012 034 156 078 59a 9bc 3de 57f 4bd 26b ace 8ef 479 13a 28c 6df

\-\ 012 034 567 589 0ab cde 6cf 136 78d 2ad 9ef 49b 37c 5bf 28e 14a
\end{config}

The first of these is one of the two flag-transitive configurations on 16 points; see~\cite{Betten2000} and Table~\ref{tab:enum}. Indeed its Levi graph is the Dyck graph, which is well-known and is the unique arc-transitive cubic graph on 32 vertices. The Dyck graph is graph F032A in the Foster census~\cite{foster88}. Although this graph has a number of known constructions, it seems that none of these can be generalised to produce further examples of configurations without small blocking sets. Unfortunately therefore, we cannot provide a construction for an infinite class of symmetric configurations $v_3$ having a minimal blocking set of maximum cardinality, and this remains a significant open problem.

\begin{table}
	\centering
	\begin{tabular}{|cr|rr|}
		\hline
		Points & Connected & \multicolumn{2}{r|}{Minimal blocking sets}\\
		$v$ & configurations & Size & Number\\
		\hline
		7 & 1 & None & 1\\
		\hline
		8 & 1 & 4 & 1\\
		\hline
		9 & 3 & 3 & 2\\
		& & 4 & 1\\
		\hline
		10 & 10 & 4 & 8\\
		& & 5 & 2\\
		\hline
		11 & 31 & 4 & 25\\
		& & 5 & 6\\
		\hline
		12 & 229 & 4 & 45\\
		& & 5 & 182\\
		& & 6 & 2\\
		\hline
		13 & 2,036 & None & 1\\
		& & 5 & 2,020\\
		& & 6 & 15\\
		\hline
		14 & 21,398 & 5 & 16,884\\
		& & 6 & 4,514\\
		& & 7 & 0\\
		\hline
		15 & 245,341 & 5 & 24,550\\
		& & 6 & 220,720\\
		& & 7 & 21\\
		\hline
		16 & 3,004,877 & 6 & 2,992,125\\
		& & 7 & 12,750\\
		& & 8 & 2\\
		\hline
		17 & 38,904,486 & 6 & 25,065,267\\
		& & 7 & 13,839,209\\
		& & 8 & 10\\
		\hline
	\end{tabular}
	\caption{Sizes of minimum blocking sets of connected configurations}
	\label{tab:bssizes}
\end{table}

However, we are able to construct symmetric configurations whose minimal blocking sets have a size as far away as we please from both the minimum or maximum possible cardinalities, as the following theorem and corollary show.
 
\begin{theorem}\label{thm:blocking013}
Let $v\geq 8$ and let $C_v$ be the cyclic configuration on $v$ points generated by the block $\{0,1,3\}$ under the mapping $i\mapsto i+1\!\pmod{v}$. Then the size of a minimal blocking set in $C_v$ is:
\[
m(v)=2\left\lfloor\frac{v}{5}\right\rfloor+\varepsilon
,\text{ where }\varepsilon=
\begin{cases}
	0&\text{ if }v\equiv 0\!\pmod{5};\\
	1&\text{ if }v\equiv 1\!\pmod{5};\\
	2&\text{ if }v\equiv 2,3,4\!\pmod{5}.\\
\end{cases}
\]
\end{theorem}
The proof of Theorem~\ref{thm:blocking013} is much simplified by transforming the problem into an equivalent problem concerning the existence of binary words. A \emph{binary word} $\mathbf{b}$ of length $n$ is a sequence $b_0,b_1,\ldots,b_{n-1}$ where each $b_i\in\{0,1\}$. We shall be concerned with \emph{circular} binary words, where the digit $b_0$ is considered to follow $b_{n-1}$; informally, the word ``wraps round'' with period $n$. A \emph{subword} of length $m$ is a sequence $b_i,b_{i+1},\ldots,b_{i+m-1}$ where the subscripts are taken mod~$n$; in other words, the subword starts at position $i$ and wraps round if necessary. The \emph{weight} $w(\mathbf{b})$ of a word $\mathbf{b}$ is simply the number of 1s in $\mathbf{b}$. A sequence of $k$ consecutive 1s in a circular binary word with 0s at either end is called a \emph{run of length} $k$; similarly for a sequence of 0s surrounded by 1s.

To make the connection with blocking sets, we let $C_v$ be a cyclic configuration as in the statement of the theorem, and identify the point set $V$ of $C_v$ with the elements $\{0,1,\ldots,v-1\}$ of the cyclic group $\Z_v$. To each subset $S\subseteq V$ we identify a binary word $\mathbf{b}(S)=b_0,b_1,\ldots,b_{v-1}$ where $b_i=1$ if $i\in S$ and $b_i=0$ otherwise. Since each block of $C_v$ has the form $\{m,m+1,m+3\}$, it is immediate that a subset $S$ is a blocking set for $C_v$ if and only if the corresponding circular binary word $\mathbf{b}(S)$ does not contain any of the subwords 0000, 0010, 1101 or 1111. The problem of finding a minimal blocking set is therefore equivalent to finding the minimum weight of a circular binary word satisfying this forbidden subword criterion. We begin with two simple lemmas.

\begin{lemma}\label{lem:weight}
Suppose $\mathbf{b}=\mathbf{b}(S)$ is a circular binary word corresponding to a blocking set $S$ of the configuration $C_v$. Then any subword of $\mathbf{b}$ of length 5 has weight 2 or 3.
\end{lemma}
\begin{proof}
Clearly any subword of length 5 and weight 0 contains the forbidden subword 0000. It is easy to see that the only possible length 5 subword of weight 1 is 01000. The digit immediately to the left of this subword must be 1, otherwise we get the forbidden subword 0010. Then the next digit to the left again must be 0, to avoid the forbidden subword 1101. Continuing in this way, we see that the sequence of digits reading leftwards from 01000 must be $1,0,1,0,1,0,\ldots$. But $\mathbf{b}$ is a circular word containing the subword 000, so this is impossible. Thus no subword of length 5 can have weight 1.

Since the roles of the binary digits 0 and 1 in this problem are symmetric (corresponding to the fact that if $S$ is a blocking set then so is its complement), it follows that no length 5 subword can have weight 4 or 5 either.
\end{proof}
\begin{lemma}\label{lem:bsform}
Suppose $\mathbf{b}=\mathbf{b}(S)$ is a circular binary word corresponding to a blocking set $S$ of the configuration $C_v$. Then $\mathbf{b}$ has one of the following forms:
\begin{enumerate}[nosep,label=(\alph*)]
	\item $01010101\ldots$ or $10101010\ldots$ (possible only if $v$ is even);
	\item a sequence of 0s and 1s in runs of length 2 or 3 only.
\end{enumerate}
\end{lemma}
\begin{proof}
The proof of Lemma~\ref{lem:weight} shows that whenever the subword 010 appears in $\mathbf{b}$, then $\mathbf{b}$ must be of type (a). A similar argument holds for the subword 101. Thus any run length of 1 forces type (a), and this is only possible if $v$ is even. Run lengths of 4 or greater are ruled out by the forbidden subwords 0000 and 1111, so the only remaining possibility is type (b). 
\end{proof}
We are now ready to complete the proof of the theorem.
\begin{proof}[Proof of Theorem~\ref{thm:blocking013}]
Suppose $\mathbf{b}=\mathbf{b}(S)$ is a circular binary word corresponding to a blocking set $S$ of the configuration $C_v$. A simple counting argument in conjunction with Lemma~\ref{lem:weight} shows that $|S|=w(\mathbf{b})\geq\frac{2v}{5}$. So writing $|S|=2\lfloor\frac{v}{5}\rfloor+\varepsilon$, it remains to find the minimum value of $\varepsilon$ in all cases. We proceed by considering all the congruence classes mod 5.

If $v\equiv 0\!\pmod{5}$, then $\mathbf{b}=11000\ 11000\ 11000 \ldots$ satisfies the conditions of Lemma~\ref{lem:bsform} and so $\varepsilon=0$.

If $v\equiv 1\!\pmod{5}$, then we know $\varepsilon\geq 1$ and $\mathbf{b}=11000\ 11000 \ldots 11000\ 1$ satisfies the conditions of Lemma~\ref{lem:bsform} and so $\varepsilon=1$.

If $v\equiv 2\!\pmod{5}$, then we know $\varepsilon\geq 1$ but an examination of all the possibilities shows that it is not possible to add a single 1 and a single 0 to a word of the form $\mathbf{b}=11000\ 11000\ 11000 \ldots$ without creating a run of length 1 or 4. Thus $\varepsilon\geq 2$, and $\mathbf{b}=111000\ 111000\ 11000 \ldots 11000$ satisfies the conditions of Lemma~\ref{lem:bsform} and so $\varepsilon=2$. 

If $v\equiv 3\!\pmod{5}$, then we know $\varepsilon\geq 2$ and $\mathbf{b}=1100\ 1100\ 11000 \ldots 11000$ satisfies the conditions of Lemma~\ref{lem:bsform} and so $\varepsilon=2$.

If $v\equiv 4\!\pmod{5}$, then we know $\varepsilon\geq 2$ and $\mathbf{b}=11000\ 1100\ 11000 \ldots 11000$ satisfies the conditions of Lemma~\ref{lem:bsform} and so $\varepsilon=2$.
\end{proof}

\begin{corollary}
Let $k\geq 1$. Then there exist:
\begin{enumerate}[nosep,label=(\alph*)]
	\item a configuration $v_3$ with a minimal blocking set of size exactly $\lceil\frac{v}{3}\rceil+k$; and
	\item a configuration $v_3$ with a minimal blocking set of size exactly $\lfloor\frac{v}{2}\rfloor-k$.
\end{enumerate}
\end{corollary}
\begin{proof}
We use Theorem~\ref{thm:blocking013}. For (a), take $v=15k$ and for (b), take $v=10k$.
\end{proof}

\subsection{Configurations without blocking sets}\label{sec:bsfree}
We now turn our attention to the case where $\chi_w=3$, i.e. to symmetric configurations with block size 3 which have no blocking set. This is an old problem which goes back some 30 years. It has appeared three times as a problem at the British Combinatorial Conference. The first time was as Problem 194 in the Proceedings of the 13th Conference~\cite{BCC13probs}, proposed by H.~ Gropp and originated by J.~W.~DiPaola and H.~Gropp. At that time there were thirteen unresolved values: 15, 16, 17, 18, 20, 23, 24, 26, 29, 30, 32, 38, 44. It appeared again as Problem 228 in the next Proceedings~\cite{BCC14probs}, by which time the five largest values had been resolved positively due to the work of Kornerup~\cite{Kornerup}. Finally in the Proceedings of the 16th Conference~\cite{BCC16probs}, Problem 333, Gropp asked whether there exists a symmetric configuration $16_3$ without a blocking set, having reported that the case of such a configuration $15_3$ had been resolved negatively.

All configurations $v_3$ for $7 \leq v \leq 18$ were enumerated by Betten, Brinkmann and Pisanski~\cite{Betten2000} in a paper published in 2000, leaving only the values $20, 23, 24, 26$ unresolved.  The problem was finally solved in 2003 by Funk et alia~\cite{funk2003det}.  A bipartite graph $G$ with bipartition $\{X, Y\}$ such that $| X | = | Y | = n$ is said to be \emph{det-extremal} if its $n\times n$ biadjacency matrix $A$ satisfies the equation $|\det(A)|=\per(A)$. (In our context, the biadjacency matrix of the Levi graph of a configuration $v_3$ is simply the $v\times v$ incidence matrix of the configuration.) Thomassen~\cite{thomassen1986sign} pointed out that a symmetric $k$-configuration is blocking set free if and only if its Levi graph is det-extremal.  In~\cite{funk2003det} the following theorem was proved from which it is an immediate corollary that there are no symmetric configurations $v_3$ without a blocking set for $v = 20, 23, 24, 26$.

\begin{theorem}[Funk, Jackson, Labbate and Sheehan]\label{thm:detextremal}
There exists a det-extremal connected cubic bipartite graph of order $2v$ if and only if $v \in\{ 7, 13, 19, 21, 22, 25\}$ or $v \geq 27$.  
\end{theorem}
The four values $20, 23, 24, 26$ indeed seem to be the most problematic.  If $v \geq 27$, it is easy to give a short self-contained account to prove that there exists a symmetric configuration $v_3$ with no blocking set and we do this below beginning with two constructions from~\cite{Bollobas1985} which we present as theorems.

\begin{theorem}[Bollob\'as and Harris]\label{thm:stitch2}
If there exist configurations $v_3$ and $(v')_3$ without a blocking set, then there exists a configuration $(v+v'-1)_3$ without a blocking set.
\end{theorem}
\begin{proof}
Denote the points and blocks of the configuration $v_3$ (resp. $(v')_3$) by $V$ and $\mathcal{B}$ (resp. $V'$ and $\mathcal{B}'$). Choose $B\in\mathcal{B}$ and suppose that points $x_1,x_2,x_3\in B$. Further choose $x'\in V'$ and suppose that $x'$ is contained in blocks $B'_1,B'_2,B'_3$. Define new blocks $B''_i=(B'_i\setminus\{x'\})\cup\{x_i\},i=1,2,3$. Then $V\cup(V'\setminus\{x'\})$ and $(\mathcal{B}\setminus\{B\})\cup(\mathcal{B}'\setminus\{B'_1,B'_2,B'_3\})\cup\{B''_1,B''_2,B''_3\}$ are the points and blocks of a configuration $(v+v'-1)_3$ which it is easily verified has no blocking set.
\end{proof}
\begin{theorem}[Bollob\'as and Harris]\label{thm:stitch3}
If there exist configurations $(v^1)_3,(v^2)_3,\ldots,(v^{2k+1})_3$ without a blocking set, then there exists a configuration $(v^1+v^2+\cdots+v^{2k+1})_3$ without a blocking set.
\end{theorem}
\begin{proof}
Denote the points and blocks of the configuration $(v^i)_3$ by $V^i$ and $\mathcal{B}^i$ respectively, $i=1,2,\ldots,2k+1$. For each $i$ choose a block $B^i\in\mathcal{B}^i$ and a point $x^i\in B^i$. Define a new block $B^i_*=(B^i\setminus \{x^i\})\cup\{x^{i+1}\}$, superscript arithmetic modulo $2k+1$. Then $\displaystyle\bigcup_{i=1}^{2k+1}V^i$ and $\displaystyle\bigcup_{i=1}^{2k+1}(\mathcal{B}^i\setminus\{B^i\})\cup\{B^i_*\}$ are the points and blocks of a configuration $(v^1+v^2+\cdots+v^{2k+1})_3$. Again it is easy to verify that this has no blocking set.
\end{proof}

We note that the construction of Theorem~\ref{thm:stitch3} was reported independently by Abbott and Hare~\cite{Abbott1999}, referencing an earlier paper of Abbott and Liu~\cite{Abbott1978}.

In order to implement the constructions we begin with three basic systems.  From~\cite{Betten2000}, in the range $7 \leq v \leq 18$ there exist only two symmetric configurations $v_3$ with no blocking set: the unique $7_3$ configuration (Fano plane) and a $13_3$ configuration obtained from two copies of it using Theorem~\ref{thm:stitch2}.

The blocks of the latter system can be represented by the following triples:

\begin{config}
\-\ 012 034 056 135 146 236 278 49c 5ab 79b 7ac 89a 8bc.
\end{config}

A symmetric configuration $22_3$ with no blocking set was given by Dorwart and Gr\"unbaum~\cite{Dorwart1992}; it is illustrated in Figure~\ref{fig:BS22} and as is evident, is obtained by merging three Fano planes. Its blocks are as follows:

\begin{config}
	\-\ 012 034 056 135 146 236 24l 58f 79c 7ak 7bl 89a 8ck 9bk abc deh dfj dgi egj eil fgh hij
\end{config}

\begin{figure}\centering
	\begin{tikzpicture}[x=0.2mm,y=-0.2mm,inner sep=0.2mm,scale=0.5,thick,vertex/.style={circle,draw,minimum size=8,fill=white}]
		\node at (245,349) [vertex,fill=white] (v1) {};
		\node at (196,374) [vertex,fill=black] (v2) {};
		\node at (140,374) [vertex,fill=white] (v3) {};
		\node at (91,158) [vertex,fill=black] (v4) {};
		\node at (57,201) [vertex,fill=white] (v5) {};
		\node at (44,254) [vertex,fill=black] (v6) {};
		\node at (57,307) [vertex,fill=white] (v7) {};
		\node at (91,349) [vertex,fill=black] (v8) {};
		\node at (291,254) [vertex,fill=white] (v9) {};
		\node at (279,201) [vertex,fill=black] (v10) {};
		\node at (245,158) [vertex,fill=white] (v11) {};
		\node at (196,134) [vertex,fill=black] (v12) {};
		\node at (140,134) [vertex,fill=white] (v13) {};
		\node at (279,307) [vertex,fill=black] (v14) {};
		\node at (459,696) [vertex,fill=black] (v15) {};
		\node at (409,720) [vertex,fill=white] (v16) {};
		\node at (355,720) [vertex,fill=black] (v17) {};
		\node at (305,504) [vertex,fill=white] (v18) {};
		\node at (270,547) [vertex,fill=black] (v19) {};
		\node at (258,600) [vertex,fill=white] (v20) {};
		\node at (270,654) [vertex,fill=black] (v21) {};
		\node at (305,696) [vertex,fill=white] (v22) {};
		\node at (505,600) [vertex,fill=black] (v23) {};
		\node at (492,547) [vertex,fill=white] (v24) {};
		\node at (459,504) [vertex,fill=black] (v25) {};
		\node at (409,480) [vertex,fill=white] (v26) {};
		\node at (355,480) [vertex,fill=black] (v27) {};
		\node at (492,654) [vertex,fill=white] (v28) {};
		\node at (658,349) [vertex,fill=white] (v29) {};
		\node at (610,374) [vertex,fill=black] (v30) {};
		\node at (555,374) [vertex,fill=white] (v31) {};
		\node at (505,158) [vertex,fill=black] (v32) {};
		\node at (470,201) [vertex,fill=white] (v33) {};
		\node at (459,254) [vertex,fill=black] (v34) {};
		\node at (470,307) [vertex,fill=white] (v35) {};
		\node at (505,349) [vertex,fill=black] (v36) {};
		\node at (705,254) [vertex,fill=white] (v37) {};
		\node at (692,201) [vertex,fill=black] (v38) {};
		\node at (658,158) [vertex,fill=white] (v39) {};
		\node at (610,134) [vertex,fill=black] (v40) {};
		\node at (555,134) [vertex,fill=white] (v41) {};
		\node at (692,307) [vertex,fill=black] (v42) {};
		\path
		(v1) edge (v2)
		(v2) edge (v3)
		(v3) edge (v8)
		(v4) edge (v5)
		(v4) edge (v13)
		(v5) edge (v6)
		(v6) edge (v7)
		(v7) edge (v8)
		(v9) edge (v10)
		(v9) edge (v14)
		(v10) edge (v11)
		(v11) edge (v12)
		(v12) edge (v13)
		(v7) edge (v14)
		(v2) edge (v5)
		(v3) edge (v10)
		(v6) edge (v11)
		(v4) edge (v9)
		(v1) edge (v12)
		(v8) edge (v13)
		(v15) edge (v16)
		(v15) edge (v28)
		(v16) edge (v17)
		(v17) edge (v22)
		(v18) edge (v19)
		(v18) edge (v27)
		(v19) edge (v20)
		(v20) edge (v21)
		(v21) edge (v22)
		(v23) edge (v24)
		(v23) edge (v28)
		(v24) edge (v25)
		(v25) edge (v26)
		(v21) edge (v28)
		(v16) edge (v19)
		(v17) edge (v24)
		(v20) edge (v25)
		(v18) edge (v23)
		(v15) edge (v26)
		(v22) edge (v27)
		(v29) edge (v30)
		(v29) edge (v42)
		(v30) edge (v31)
		(v31) edge (v36)
		(v32) edge (v33)
		(v32) edge (v41)
		(v33) edge (v34)
		(v34) edge (v35)
		(v37) edge (v38)
		(v37) edge (v42)
		(v38) edge (v39)
		(v39) edge (v40)
		(v40) edge (v41)
		(v35) edge (v42)
		(v30) edge (v33)
		(v31) edge (v38)
		(v34) edge (v39)
		(v32) edge (v37)
		(v29) edge (v40)
		(v36) edge (v41)
		(v1) edge (v27)
		(v36) edge (v26)
		(v35) edge (v14)
		;
	\end{tikzpicture}
	\quad
		\begin{tikzpicture}[x=0.2mm,y=-0.2mm,inner sep=0.2mm,scale=0.5,thick,vertex/.style={circle,draw,minimum size=8,fill=white}]
		\node at (245,349) [vertex,fill=white] (v1) {};
		\node at (196,374) [vertex,fill=black] (v2) {};
		\node at (140,374) [vertex,fill=white] (v3) {};
		\node at (91,158) [vertex,fill=black] (v4) {};
		\node at (57,201) [vertex,fill=white] (v5) {};
		\node at (44,254) [vertex,fill=black] (v6) {};
		\node at (57,307) [vertex,fill=white] (v7) {};
		\node at (91,349) [vertex,fill=black] (v8) {};
		\node at (291,254) [vertex,fill=white] (v9) {};
		\node at (279,201) [vertex,fill=black] (v10) {};
		\node at (245,158) [vertex,fill=white] (v11) {};
		\node at (196,134) [vertex,fill=black] (v12) {};
		\node at (140,134) [vertex,fill=white] (v13) {};
		\node at (279,307) [vertex,fill=black] (v14) {};
		\node at (459,696) [vertex,fill=black] (v15) {};
		\node at (409,720) [vertex,fill=white] (v16) {};
		\node at (355,720) [vertex,fill=black] (v17) {};
		\node at (305,504) [vertex,fill=white] (v18) {};
		\node at (270,547) [vertex,fill=black] (v19) {};
		\node at (258,600) [vertex,fill=white] (v20) {};
		\node at (270,654) [vertex,fill=black] (v21) {};
		\node at (305,696) [vertex,fill=white] (v22) {};
		\node at (505,600) [vertex,fill=black] (v23) {};
		\node at (492,547) [vertex,fill=white] (v24) {};
		\node at (459,504) [vertex,fill=black] (v25) {};
		\node at (409,480) [vertex,fill=white] (v26) {};
		\node at (355,480) [vertex,fill=black] (v27) {};
		\node at (492,654) [vertex,fill=white] (v28) {};
		\node at (658,349) [vertex,fill=white] (v29) {};
		\node at (610,374) [vertex,fill=black] (v30) {};
		\node at (555,374) [vertex,fill=white] (v31) {};
		\node at (505,158) [vertex,fill=black] (v32) {};
		\node at (470,201) [vertex,fill=white] (v33) {};
		\node at (459,254) [vertex,fill=black] (v34) {};
		\node at (470,307) [vertex,fill=white] (v35) {};
		\node at (505,349) [vertex,fill=black] (v36) {};
		\node at (705,254) [vertex,fill=white] (v37) {};
		\node at (692,201) [vertex,fill=black] (v38) {};
		\node at (658,158) [vertex,fill=white] (v39) {};
		\node at (610,134) [vertex,fill=black] (v40) {};
		\node at (555,134) [vertex,fill=white] (v41) {};
		\node at (692,307) [vertex,fill=black] (v42) {};
		\node at (433,382) [vertex,fill=white] (v43) {};
		\node at (323,382) [vertex,fill=black] (v44) {};
		\path
		(v1) edge (v2)
		(v2) edge (v3)
		(v3) edge (v8)
		(v4) edge (v5)
		(v4) edge (v13)
		(v5) edge (v6)
		(v6) edge (v7)
		(v7) edge (v8)
		(v9) edge (v10)
		(v9) edge (v14)
		(v10) edge (v11)
		(v11) edge (v12)
		(v12) edge (v13)
		(v7) edge (v14)
		(v2) edge (v5)
		(v3) edge (v10)
		(v6) edge (v11)
		(v4) edge (v9)
		(v1) edge (v12)
		(v8) edge (v13)
		(v15) edge (v16)
		(v15) edge (v28)
		(v16) edge (v17)
		(v17) edge (v22)
		(v18) edge (v19)
		(v18) edge (v27)
		(v19) edge (v20)
		(v20) edge (v21)
		(v21) edge (v22)
		(v23) edge (v24)
		(v23) edge (v28)
		(v24) edge (v25)
		(v25) edge (v26)
		(v21) edge (v28)
		(v16) edge (v19)
		(v17) edge (v24)
		(v20) edge (v25)
		(v18) edge (v23)
		(v15) edge (v26)
		(v22) edge (v27)
		(v29) edge (v30)
		(v29) edge (v42)
		(v30) edge (v31)
		(v31) edge (v36)
		(v32) edge (v33)
		(v32) edge (v41)
		(v33) edge (v34)
		(v34) edge (v35)
		(v37) edge (v38)
		(v37) edge (v42)
		(v38) edge (v39)
		(v39) edge (v40)
		(v40) edge (v41)
		(v35) edge (v42)
		(v30) edge (v33)
		(v31) edge (v38)
		(v34) edge (v39)
		(v32) edge (v37)
		(v29) edge (v40)
		(v36) edge (v41)
		(v1) edge (v44)
		(v36) edge (v43)
		(v35) edge (v44)
		(v26) edge (v44)
		(v27) edge (v43)
		(v14) edge (v43)
		;
	\end{tikzpicture}
	\caption{Levi graphs of the unique configurations $21_3$ and $22_3$ with no blocking set}
	\label{fig:BS22}
\end{figure}
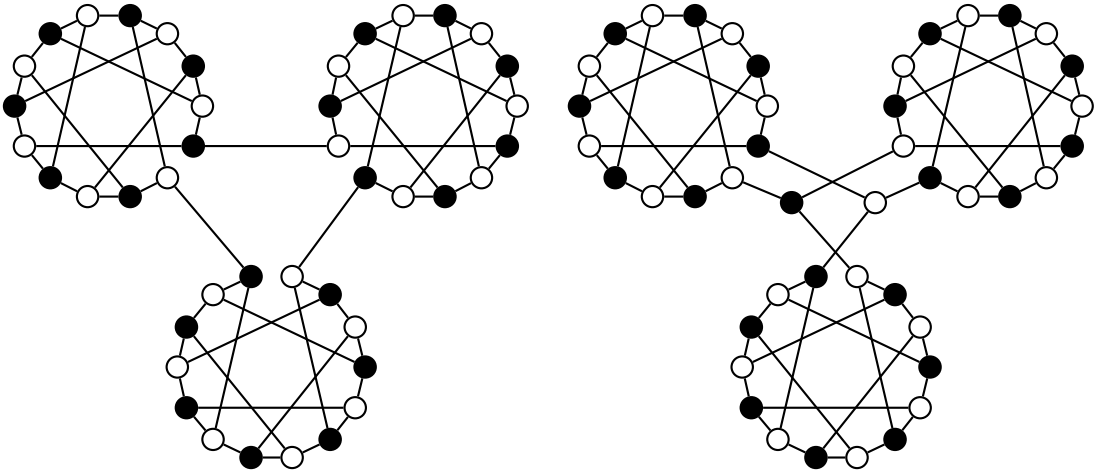

To deduce the existence of blocking set free configurations $v_3$ for all $v\geq 27$, first note that by putting $v'=7$ in Theorem~\ref{thm:stitch2}, it follows that if there exists a blocking set free configuration $v_3$ then there exists a blocking set free configuration $(v+6)_3$. Thus once there exist such configurations for six consecutive values of $v$, existence for all larger values of $v$ follows inductively. Existence for the value $v=27$ follows from Theorem~\ref{thm:stitch3} by putting $v^1=v^2=7$ and $v^3=13$, and for the value $v=28$ from Theorem~\ref{thm:stitch2} by putting $v=7$ and $v'=22$. For the value $v=31$, first construct a configuration $25_3$ from Theorem~\ref{thm:stitch2} by putting $v = v' = 13$ and then, again from Theorem~\ref{thm:stitch2}, by putting $v = 7$ and $v' = 25$.

As reported in~\cite{Gropp1997}, Kornerup~\cite{Kornerup} constructed blocking set free configurations $v_3$ for the values $v=29,30,32$. These are contained in a thesis of the University of Aarhus which we have been unable to see, and the configurations found do not seem to be published elsewhere. Thus in order to give a complete account in one place, we have also constructed configurations $29_3$, $30_3$ and $32_3$ without a blocking set. We show their Levi graphs in Figure~\ref{fig:BS2932}, and include the blocks below. 

A blocking set free configuration $29_3$:

\begin{config}
	\-\ 012 034 056 135 146 29d 2bc 367 457 7es 89b 8as 8cd 9ac abd egk eij fgi\\
	\-\ fhl fjk ghj hik lnr lpq mnp mos mqr noq opr
\end{config}

A blocking set free configuration $30_3$:

\begin{config}
	\-\ 012 034 056 135 146 2bt 2fm 367 457 789 8ae 8cd 9ac 9de abd bce fhl fjk\\
	\-\ ghj git gkl hik ijl mos mqr noq npt nrs opr pqs
\end{config}

A blocking set free configuration $32_3$:

\begin{config}
	\-\ 012 034 056 135 146 27k 2tv 36l 45l 78a 7bc 89b 8ck 9ac 9gt abk deu dfj\\
	\-\ dhi efh eij fgi ghj luv mnv mos mqr noq nrs opr pqs ptu
\end{config}

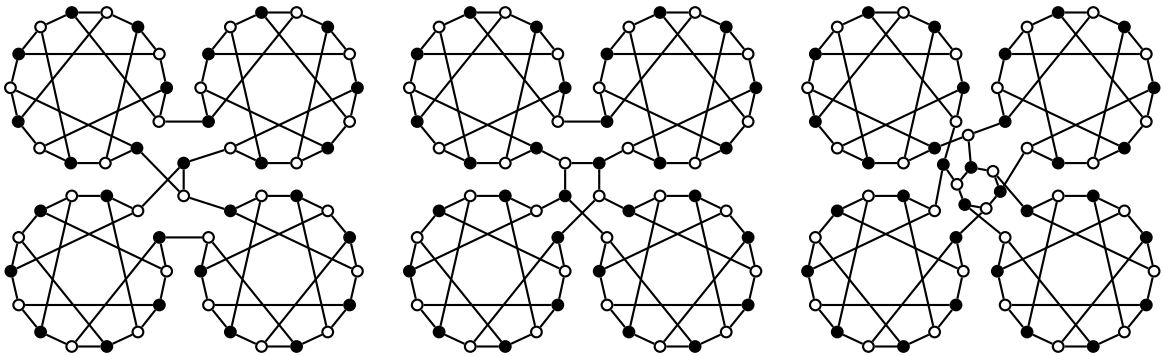
\begin{figure}\centering
	\begin{tikzpicture}[x=0.2mm,y=-0.2mm,inner sep=0.2mm,scale=0.32,thick,vertex/.style={circle,draw,minimum size=4,fill=white}]
		\node at (339,270) [vertex,fill=white] (v1) {$$};
		\node at (294,325) [vertex,fill=black] (v2) {$$};
		\node at (229,356) [vertex,fill=white] (v3) {$$};
		\node at (51,130) [vertex,fill=black] (v4) {$$};
		\node at (34,200) [vertex,fill=white] (v5) {$$};
		\node at (50,270) [vertex,fill=black] (v6) {$$};
		\node at (94,325) [vertex,fill=white] (v7) {$$};
		\node at (158,356) [vertex,fill=black] (v8) {$$};
		\node at (340,130) [vertex,fill=white] (v9) {$$};
		\node at (296,74) [vertex,fill=black] (v10) {$$};
		\node at (232,44) [vertex,fill=white] (v11) {$$};
		\node at (160,44) [vertex,fill=black] (v12) {$$};
		\node at (96,74) [vertex,fill=white] (v13) {$$};
		\node at (355,200) [vertex,fill=black] (v14) {$$};
		\node at (296,706) [vertex,fill=white] (v15) {$$};
		\node at (232,736) [vertex,fill=black] (v16) {$$};
		\node at (160,736) [vertex,fill=white] (v17) {$$};
		\node at (96,455) [vertex,fill=black] (v18) {$$};
		\node at (51,510) [vertex,fill=white] (v19) {$$};
		\node at (35,580) [vertex,fill=black] (v20) {$$};
		\node at (51,650) [vertex,fill=white] (v21) {$$};
		\node at (96,706) [vertex,fill=black] (v22) {$$};
		\node at (355,580) [vertex,fill=white] (v23) {$$};
		\node at (340,510) [vertex,fill=black] (v24) {$$};
		\node at (296,455) [vertex,fill=white] (v25) {$$};
		\node at (232,424) [vertex,fill=black] (v26) {$$};
		\node at (160,424) [vertex,fill=white] (v27) {$$};
		\node at (340,650) [vertex,fill=black] (v28) {$$};
		\node at (731,270) [vertex,fill=white] (v29) {$$};
		\node at (686,325) [vertex,fill=black] (v30) {$$};
		\node at (622,356) [vertex,fill=white] (v31) {$$};
		\node at (441,130) [vertex,fill=black] (v32) {$$};
		\node at (425,200) [vertex,fill=white] (v33) {$$};
		\node at (441,270) [vertex,fill=black] (v34) {$$};
		\node at (486,325) [vertex,fill=white] (v35) {$$};
		\node at (550,356) [vertex,fill=black] (v36) {$$};
		\node at (731,130) [vertex,fill=white] (v37) {$$};
		\node at (686,74) [vertex,fill=black] (v38) {$$};
		\node at (622,44) [vertex,fill=white] (v39) {$$};
		\node at (550,44) [vertex,fill=black] (v40) {$$};
		\node at (486,74) [vertex,fill=white] (v41) {$$};
		\node at (747,200) [vertex,fill=black] (v42) {$$};
		\node at (747,580) [vertex,fill=white] (v43) {$$};
		\node at (731,650) [vertex,fill=black] (v44) {$$};
		\node at (686,706) [vertex,fill=white] (v45) {$$};
		\node at (425,580) [vertex,fill=black] (v46) {$$};
		\node at (441,650) [vertex,fill=white] (v47) {$$};
		\node at (486,706) [vertex,fill=black] (v48) {$$};
		\node at (550,736) [vertex,fill=white] (v49) {$$};
		\node at (622,736) [vertex,fill=black] (v50) {$$};
		\node at (686,455) [vertex,fill=white] (v51) {$$};
		\node at (622,424) [vertex,fill=black] (v52) {$$};
		\node at (550,424) [vertex,fill=white] (v53) {$$};
		\node at (486,455) [vertex,fill=black] (v54) {$$};
		\node at (441,510) [vertex,fill=white] (v55) {$$};
		\node at (731,510) [vertex,fill=black] (v56) {$$};
		\node at (390,356) [vertex,fill=black] (v57) {$$};
		\node at (390,424) [vertex,fill=white] (v58) {$$};
		\path
		(v2) edge (v3)
		(v3) edge (v8)
		(v4) edge (v5)
		(v4) edge (v13)
		(v5) edge (v6)
		(v6) edge (v7)
		(v7) edge (v8)
		(v9) edge (v10)
		(v9) edge (v14)
		(v10) edge (v11)
		(v11) edge (v12)
		(v12) edge (v13)
		(v7) edge (v14)
		(v2) edge (v5)
		(v3) edge (v10)
		(v6) edge (v11)
		(v4) edge (v9)
		(v8) edge (v13)
		(v15) edge (v16)
		(v15) edge (v28)
		(v16) edge (v17)
		(v17) edge (v22)
		(v18) edge (v19)
		(v18) edge (v27)
		(v19) edge (v20)
		(v20) edge (v21)
		(v21) edge (v22)
		(v23) edge (v24)
		(v23) edge (v28)
		(v26) edge (v27)
		(v21) edge (v28)
		(v16) edge (v19)
		(v17) edge (v24)
		(v18) edge (v23)
		(v15) edge (v26)
		(v22) edge (v27)
		(v29) edge (v30)
		(v29) edge (v42)
		(v30) edge (v31)
		(v31) edge (v36)
		(v32) edge (v33)
		(v32) edge (v41)
		(v33) edge (v34)
		(v37) edge (v38)
		(v37) edge (v42)
		(v38) edge (v39)
		(v39) edge (v40)
		(v40) edge (v41)
		(v30) edge (v33)
		(v31) edge (v38)
		(v34) edge (v39)
		(v32) edge (v37)
		(v29) edge (v40)
		(v36) edge (v41)
		(v43) edge (v44)
		(v43) edge (v56)
		(v44) edge (v45)
		(v45) edge (v50)
		(v46) edge (v47)
		(v47) edge (v48)
		(v48) edge (v49)
		(v49) edge (v50)
		(v51) edge (v52)
		(v51) edge (v56)
		(v52) edge (v53)
		(v53) edge (v54)
		(v49) edge (v56)
		(v44) edge (v47)
		(v45) edge (v52)
		(v48) edge (v53)
		(v46) edge (v51)
		(v43) edge (v54)
		(v1) edge (v12)
		(v1) edge (v34)
		(v35) edge (v36)
		(v35) edge (v42)
		(v35) edge (v57)
		(v50) edge (v55)
		(v24) edge (v55)
		(v25) edge (v26)
		(v25) edge (v57)
		(v57) edge (v58)
		(v20) edge (v25)
		(v46) edge (v55)
		(v54) edge (v58)
		(v1) edge (v14)
		(v2) edge (v58)
		;
	\end{tikzpicture}
	\quad
	\begin{tikzpicture}[x=0.2mm,y=-0.2mm,inner sep=0.2mm,scale=0.32,thick,vertex/.style={circle,draw,minimum size=4,fill=white}]
		\node at (340,270) [vertex,fill=white] (v1) {$$};
		\node at (296,325) [vertex,fill=black] (v2) {$$};
		\node at (232,356) [vertex,fill=white] (v3) {$$};
		\node at (51,130) [vertex,fill=black] (v4) {$$};
		\node at (35,200) [vertex,fill=white] (v5) {$$};
		\node at (51,270) [vertex,fill=black] (v6) {$$};
		\node at (96,325) [vertex,fill=white] (v7) {$$};
		\node at (160,356) [vertex,fill=black] (v8) {$$};
		\node at (340,130) [vertex,fill=white] (v9) {$$};
		\node at (296,74) [vertex,fill=black] (v10) {$$};
		\node at (232,44) [vertex,fill=white] (v11) {$$};
		\node at (160,44) [vertex,fill=black] (v12) {$$};
		\node at (96,74) [vertex,fill=white] (v13) {$$};
		\node at (355,200) [vertex,fill=black] (v14) {$$};
		\node at (296,706) [vertex,fill=white] (v15) {$$};
		\node at (232,736) [vertex,fill=black] (v16) {$$};
		\node at (160,736) [vertex,fill=white] (v17) {$$};
		\node at (96,455) [vertex,fill=black] (v18) {$$};
		\node at (51,510) [vertex,fill=white] (v19) {$$};
		\node at (35,580) [vertex,fill=black] (v20) {$$};
		\node at (51,650) [vertex,fill=white] (v21) {$$};
		\node at (96,706) [vertex,fill=black] (v22) {$$};
		\node at (355,580) [vertex,fill=white] (v23) {$$};
		\node at (340,510) [vertex,fill=black] (v24) {$$};
		\node at (296,455) [vertex,fill=white] (v25) {$$};
		\node at (232,424) [vertex,fill=black] (v26) {$$};
		\node at (160,424) [vertex,fill=white] (v27) {$$};
		\node at (340,650) [vertex,fill=black] (v28) {$$};
		\node at (731,270) [vertex,fill=white] (v29) {$$};
		\node at (686,325) [vertex,fill=black] (v30) {$$};
		\node at (622,356) [vertex,fill=white] (v31) {$$};
		\node at (441,130) [vertex,fill=black] (v32) {$$};
		\node at (425,200) [vertex,fill=white] (v33) {$$};
		\node at (441,270) [vertex,fill=black] (v34) {$$};
		\node at (484,325) [vertex,fill=white] (v35) {$$};
		\node at (550,356) [vertex,fill=black] (v36) {$$};
		\node at (731,130) [vertex,fill=white] (v37) {$$};
		\node at (686,74) [vertex,fill=black] (v38) {$$};
		\node at (622,44) [vertex,fill=white] (v39) {$$};
		\node at (550,44) [vertex,fill=black] (v40) {$$};
		\node at (486,74) [vertex,fill=white] (v41) {$$};
		\node at (747,200) [vertex,fill=black] (v42) {$$};
		\node at (747,580) [vertex,fill=white] (v43) {$$};
		\node at (731,650) [vertex,fill=black] (v44) {$$};
		\node at (686,706) [vertex,fill=white] (v45) {$$};
		\node at (425,580) [vertex,fill=black] (v46) {$$};
		\node at (441,650) [vertex,fill=white] (v47) {$$};
		\node at (486,706) [vertex,fill=black] (v48) {$$};
		\node at (550,736) [vertex,fill=white] (v49) {$$};
		\node at (622,736) [vertex,fill=black] (v50) {$$};
		\node at (686,455) [vertex,fill=white] (v51) {$$};
		\node at (622,424) [vertex,fill=black] (v52) {$$};
		\node at (550,424) [vertex,fill=white] (v53) {$$};
		\node at (486,455) [vertex,fill=black] (v54) {$$};
		\node at (441,510) [vertex,fill=white] (v55) {$$};
		\node at (731,510) [vertex,fill=black] (v56) {$$};
		\node at (425,356) [vertex,fill=black] (v57) {$$};
		\node at (355,356) [vertex,fill=white] (v58) {$$};
		\node at (425,424) [vertex,fill=white] (v59) {$$};
		\node at (355,424) [vertex,fill=black] (v60) {$$};
		\path
		(v2) edge (v3)
		(v3) edge (v8)
		(v4) edge (v5)
		(v4) edge (v13)
		(v5) edge (v6)
		(v6) edge (v7)
		(v7) edge (v8)
		(v9) edge (v10)
		(v9) edge (v14)
		(v10) edge (v11)
		(v11) edge (v12)
		(v12) edge (v13)
		(v7) edge (v14)
		(v2) edge (v5)
		(v3) edge (v10)
		(v6) edge (v11)
		(v4) edge (v9)
		(v8) edge (v13)
		(v15) edge (v16)
		(v15) edge (v28)
		(v16) edge (v17)
		(v17) edge (v22)
		(v18) edge (v19)
		(v18) edge (v27)
		(v19) edge (v20)
		(v20) edge (v21)
		(v21) edge (v22)
		(v23) edge (v24)
		(v23) edge (v28)
		(v26) edge (v27)
		(v21) edge (v28)
		(v16) edge (v19)
		(v17) edge (v24)
		(v18) edge (v23)
		(v15) edge (v26)
		(v22) edge (v27)
		(v29) edge (v30)
		(v29) edge (v42)
		(v30) edge (v31)
		(v31) edge (v36)
		(v32) edge (v33)
		(v32) edge (v41)
		(v33) edge (v34)
		(v37) edge (v38)
		(v37) edge (v42)
		(v38) edge (v39)
		(v39) edge (v40)
		(v40) edge (v41)
		(v30) edge (v33)
		(v31) edge (v38)
		(v34) edge (v39)
		(v32) edge (v37)
		(v29) edge (v40)
		(v36) edge (v41)
		(v43) edge (v44)
		(v43) edge (v56)
		(v44) edge (v45)
		(v45) edge (v50)
		(v46) edge (v47)
		(v47) edge (v48)
		(v48) edge (v49)
		(v49) edge (v50)
		(v51) edge (v52)
		(v51) edge (v56)
		(v52) edge (v53)
		(v53) edge (v54)
		(v49) edge (v56)
		(v44) edge (v47)
		(v45) edge (v52)
		(v48) edge (v53)
		(v46) edge (v51)
		(v43) edge (v54)
		(v1) edge (v12)
		(v35) edge (v36)
		(v35) edge (v42)
		(v50) edge (v55)
		(v25) edge (v26)
		(v20) edge (v25)
		(v46) edge (v55)
		(v1) edge (v14)
		(v1) edge (v34)
		(v57) edge (v58)
		(v58) edge (v60)
		(v2) edge (v58)
		(v35) edge (v57)
		(v55) edge (v60)
		(v25) edge (v60)
		(v54) edge (v59)
		(v57) edge (v59)
		(v24) edge (v59)
		;
	\end{tikzpicture}
	\quad
	\begin{tikzpicture}[x=0.2mm,y=-0.2mm,inner sep=0.2mm,scale=0.32,thick,vertex/.style={circle,draw,minimum size=4,fill=white}]
		\node at (340,270) [vertex,fill=white] (v1) {$$};
		\node at (296,325) [vertex,fill=black] (v2) {$$};
		\node at (232,356) [vertex,fill=white] (v3) {$$};
		\node at (51,130) [vertex,fill=black] (v4) {$$};
		\node at (35,200) [vertex,fill=white] (v5) {$$};
		\node at (51,270) [vertex,fill=black] (v6) {$$};
		\node at (96,325) [vertex,fill=white] (v7) {$$};
		\node at (160,356) [vertex,fill=black] (v8) {$$};
		\node at (340,130) [vertex,fill=white] (v9) {$$};
		\node at (296,74) [vertex,fill=black] (v10) {$$};
		\node at (232,44) [vertex,fill=white] (v11) {$$};
		\node at (160,44) [vertex,fill=black] (v12) {$$};
		\node at (96,74) [vertex,fill=white] (v13) {$$};
		\node at (355,200) [vertex,fill=black] (v14) {$$};
		\node at (686,455) [vertex,fill=white] (v15) {$$};
		\node at (731,510) [vertex,fill=black] (v16) {$$};
		\node at (747,580) [vertex,fill=white] (v17) {$$};
		\node at (486,706) [vertex,fill=black] (v18) {$$};
		\node at (550,736) [vertex,fill=white] (v19) {$$};
		\node at (622,736) [vertex,fill=black] (v20) {$$};
		\node at (686,706) [vertex,fill=white] (v21) {$$};
		\node at (731,650) [vertex,fill=black] (v22) {$$};
		\node at (550,424) [vertex,fill=white] (v23) {$$};
		\node at (486,455) [vertex,fill=black] (v24) {$$};
		\node at (441,510) [vertex,fill=white] (v25) {$$};
		\node at (425,580) [vertex,fill=black] (v26) {$$};
		\node at (441,650) [vertex,fill=white] (v27) {$$};
		\node at (622,424) [vertex,fill=black] (v28) {$$};
		\node at (731,270) [vertex,fill=white] (v29) {$$};
		\node at (686,325) [vertex,fill=black] (v30) {$$};
		\node at (622,356) [vertex,fill=white] (v31) {$$};
		\node at (441,130) [vertex,fill=black] (v32) {$$};
		\node at (425,200) [vertex,fill=white] (v33) {$$};
		\node at (441,270) [vertex,fill=black] (v34) {$$};
		\node at (486,325) [vertex,fill=white] (v35) {$$};
		\node at (550,356) [vertex,fill=black] (v36) {$$};
		\node at (731,130) [vertex,fill=white] (v37) {$$};
		\node at (686,74) [vertex,fill=black] (v38) {$$};
		\node at (622,44) [vertex,fill=white] (v39) {$$};
		\node at (550,44) [vertex,fill=black] (v40) {$$};
		\node at (486,74) [vertex,fill=white] (v41) {$$};
		\node at (747,200) [vertex,fill=black] (v42) {$$};
		\node at (160,736) [vertex,fill=white] (v43) {$$};
		\node at (96,706) [vertex,fill=black] (v44) {$$};
		\node at (51,650) [vertex,fill=white] (v45) {$$};
		\node at (232,424) [vertex,fill=black] (v46) {$$};
		\node at (160,424) [vertex,fill=white] (v47) {$$};
		\node at (96,455) [vertex,fill=black] (v48) {$$};
		\node at (51,510) [vertex,fill=white] (v49) {$$};
		\node at (35,580) [vertex,fill=black] (v50) {$$};
		\node at (296,706) [vertex,fill=white] (v51) {$$};
		\node at (340,650) [vertex,fill=black] (v52) {$$};
		\node at (355,580) [vertex,fill=white] (v53) {$$};
		\node at (340,510) [vertex,fill=black] (v54) {$$};
		\node at (296,455) [vertex,fill=white] (v55) {$$};
		\node at (232,736) [vertex,fill=black] (v56) {$$};
		\node at (431,415) [vertex,fill=black] (v57) {$$};
		\node at (342,400) [vertex,fill=white] (v58) {$$};
		\node at (402,450) [vertex,fill=white] (v59) {$$};
		\node at (358,442) [vertex,fill=black] (v60) {$$};
		\node at (416,373) [vertex,fill=white] (v61) {$$};
		\node at (371,365) [vertex,fill=black] (v62) {$$};
		\node at (314,359) [vertex,fill=black] (v63) {$$};
		\node at (365,298) [vertex,fill=white] (v64) {$$};
		\path
		(v2) edge (v3)
		(v3) edge (v8)
		(v4) edge (v5)
		(v4) edge (v13)
		(v5) edge (v6)
		(v6) edge (v7)
		(v7) edge (v8)
		(v9) edge (v10)
		(v9) edge (v14)
		(v10) edge (v11)
		(v11) edge (v12)
		(v12) edge (v13)
		(v7) edge (v14)
		(v2) edge (v5)
		(v3) edge (v10)
		(v6) edge (v11)
		(v4) edge (v9)
		(v8) edge (v13)
		(v15) edge (v16)
		(v15) edge (v28)
		(v16) edge (v17)
		(v17) edge (v22)
		(v18) edge (v19)
		(v18) edge (v27)
		(v19) edge (v20)
		(v20) edge (v21)
		(v21) edge (v22)
		(v23) edge (v24)
		(v23) edge (v28)
		(v26) edge (v27)
		(v21) edge (v28)
		(v16) edge (v19)
		(v17) edge (v24)
		(v18) edge (v23)
		(v15) edge (v26)
		(v22) edge (v27)
		(v29) edge (v30)
		(v29) edge (v42)
		(v30) edge (v31)
		(v31) edge (v36)
		(v32) edge (v33)
		(v32) edge (v41)
		(v33) edge (v34)
		(v37) edge (v38)
		(v37) edge (v42)
		(v38) edge (v39)
		(v39) edge (v40)
		(v40) edge (v41)
		(v30) edge (v33)
		(v31) edge (v38)
		(v34) edge (v39)
		(v32) edge (v37)
		(v29) edge (v40)
		(v36) edge (v41)
		(v43) edge (v44)
		(v43) edge (v56)
		(v44) edge (v45)
		(v45) edge (v50)
		(v46) edge (v47)
		(v47) edge (v48)
		(v48) edge (v49)
		(v49) edge (v50)
		(v51) edge (v52)
		(v51) edge (v56)
		(v52) edge (v53)
		(v53) edge (v54)
		(v49) edge (v56)
		(v44) edge (v47)
		(v45) edge (v52)
		(v48) edge (v53)
		(v46) edge (v51)
		(v43) edge (v54)
		(v1) edge (v12)
		(v35) edge (v36)
		(v35) edge (v42)
		(v50) edge (v55)
		(v25) edge (v26)
		(v20) edge (v25)
		(v46) edge (v55)
		(v1) edge (v14)
		(v1) edge (v63)
		(v35) edge (v57)
		(v55) edge (v63)
		(v25) edge (v60)
		(v62) edge (v64)
		(v34) edge (v64)
		(v2) edge (v64)
		(v54) edge (v59)
		(v57) edge (v59)
		(v59) edge (v60)
		(v61) edge (v62)
		(v57) edge (v61)
		(v24) edge (v61)
		(v58) edge (v60)
		(v58) edge (v63)
		(v58) edge (v62)
		;
	\end{tikzpicture}
	\caption{Levi graphs of configurations $29_3$, $30_3$ and $32_3$ with no blocking set}
	\label{fig:BS2932}
\end{figure}

Again we have used the ``merging'' technique and claim no originality for these. They may very well be the same systems discovered by Kornerup. 

\subsection{Connectivity of configurations}\label{sec:conn}
Here we introduce the idea of the connectivity of a symmetric configuration and derive some results. First recall that in a cubic graph, the vertex connectivity is equal to the edge connectivity. Further if a connected cubic graph is also bipartite, then the connectivity cannot be 1 and so is equal to either 2 or 3. Define the \emph{connectivity} of a symmetric configuration to be the connectivity of its Levi graph. In~\cite{funk2003det}, Funk et alia present the following operation. Let $G_1$ and $G_2$ be cubic bipartite graphs which are disjoint, and let $y\in V(G_1)$ with neighbour set $\{x_1,x_2,x_3\}$ and $x\in V(G_2)$ with neighbour set $\{y_1,y_2,y_3\}$. Then the graph 

$\quad G=(G_1\setminus y)\cup(G_2\setminus x)\cup\{x_1 y_1,x_2 y_2,x_3 y_3\}$ 

is said to be a \emph{vertex-sum} of $G_1$ and $G_2$. They then quote the following theorem which they attribute to McCuaig~\cite{mccuaig2000even}.

\begin{theorem}[McCuaig]\label{thm:mccuaig}
A 3-connected cubic bipartite graph is det-extremal if and only if it can be obtained from the Heawood graph by repeatedly applying the vertex-sum operation.
\end{theorem}

For our purposes, the significance of the vertex-sum operation on cubic bipartite graphs is that it is equivalent to the $v+v'-1$ construction of Bollob\'as and Harris given in Theorem~\ref{thm:stitch2}. Thus we have the following result.

\begin{theorem}\label{thm:3conn}
A 3-connected symmetric configuration $v_3$ without a blocking set exists if and only if $v\equiv 1\!\pmod{6}$. Moreover, such systems can only be obtained from the Fano plane by repeatedly applying the $v+v'-1$ construction.
\end{theorem}

This naturally raises the question of the spectrum of 2-connected symmetric configurations without a blocking set. From our account above it is clear that the systems $v_3$ with $v\equiv 1\!\pmod{6}$ arising from Theorem~\ref{thm:3conn} are 3-connected. There are no 2-connected systems for $v\in\{7,13,19\}$ since all have been enumerated and arise from Theorem~\ref{thm:3conn}; see Table~\ref{tab:enum} and the discussion in Section~\ref{sec:enum}. So to complete the spectrum, what is needed is a 2-connected configuration $25_3$ without a blocking set. Such a configuration does exist and its Levi graph is shown in Figure~\ref{fig:BS25}. The blocks are listed below.

\begin{config}
	\-\ 012 034 056 135 146 236 24m 5ln 78d 79c 7an 89b 8ac 9ad bcd blo ego ehi \\
	\-\ ejk fhk fij flm ghj gik mno
\end{config}

\begin{figure}\centering
	\begin{tikzpicture}[x=0.2mm,y=-0.2mm,inner sep=0.2mm,scale=0.7,thick,vertex/.style={circle,draw,minimum size=8,fill=white}]
		\node at (153,452) [vertex,fill=white] (v1) {};
		\node at (120,468) [vertex,fill=black] (v2) {};
		\node at (83,468) [vertex,fill=white] (v3) {};
		\node at (49,323) [vertex,fill=black] (v4) {};
		\node at (26,352) [vertex,fill=white] (v5) {};
		\node at (18,388) [vertex,fill=black] (v6) {};
		\node at (27,424) [vertex,fill=white] (v7) {};
		\node at (50,452) [vertex,fill=black] (v8) {};
		\node at (184,387) [vertex,fill=white] (v9) {};
		\node at (176,351) [vertex,fill=black] (v10) {};
		\node at (153,322) [vertex,fill=white] (v11) {};
		\node at (120,306) [vertex,fill=black] (v12) {};
		\node at (82,307) [vertex,fill=white] (v13) {};
		\node at (176,423) [vertex,fill=black] (v14) {};
		\node at (550,452) [vertex,fill=white] (v15) {};
		\node at (517,469) [vertex,fill=black] (v16) {};
		\node at (480,469) [vertex,fill=white] (v17) {};
		\node at (446,323) [vertex,fill=black] (v18) {};
		\node at (423,352) [vertex,fill=white] (v19) {};
		\node at (415,388) [vertex,fill=black] (v20) {};
		\node at (424,424) [vertex,fill=white] (v21) {};
		\node at (447,453) [vertex,fill=black] (v22) {};
		\node at (582,387) [vertex,fill=white] (v23) {};
		\node at (573,351) [vertex,fill=black] (v24) {};
		\node at (550,323) [vertex,fill=white] (v25) {};
		\node at (517,307) [vertex,fill=black] (v26) {};
		\node at (479,307) [vertex,fill=white] (v27) {};
		\node at (573,424) [vertex,fill=black] (v28) {};
		\node at (354,186) [vertex,fill=white] (v29) {};
		\node at (321,202) [vertex,fill=black] (v30) {};
		\node at (283,202) [vertex,fill=white] (v31) {};
		\node at (249,56) [vertex,fill=black] (v32) {};
		\node at (226,85) [vertex,fill=white] (v33) {};
		\node at (218,121) [vertex,fill=black] (v34) {};
		\node at (227,158) [vertex,fill=white] (v35) {};
		\node at (250,186) [vertex,fill=black] (v36) {};
		\node at (385,121) [vertex,fill=white] (v37) {};
		\node at (376,84) [vertex,fill=black] (v38) {};
		\node at (353,56) [vertex,fill=white] (v39) {};
		\node at (320,40) [vertex,fill=black] (v40) {};
		\node at (282,40) [vertex,fill=white] (v41) {};
		\node at (377,157) [vertex,fill=black] (v42) {};
		\node at (393,327) [vertex,fill=black] (v43) {};
		\node at (189,298) [vertex,fill=black] (v44) {};
		\node at (301,280) [vertex,fill=white] (v45) {};
		\node at (415,294) [vertex,fill=white] (v46) {};
		\node at (209,329) [vertex,fill=white] (v47) {};
		\node at (321,243) [vertex,fill=white] (v48) {};
		\node at (283,243) [vertex,fill=black] (v49) {};
		\node at (320,309) [vertex,fill=black] (v50) {};
		\path
		(v1) edge (v2)
		(v1) edge (v14)
		(v2) edge (v3)
		(v3) edge (v8)
		(v4) edge (v5)
		(v4) edge (v13)
		(v5) edge (v6)
		(v6) edge (v7)
		(v7) edge (v8)
		(v9) edge (v10)
		(v9) edge (v14)
		(v11) edge (v12)
		(v12) edge (v13)
		(v7) edge (v14)
		(v2) edge (v5)
		(v3) edge (v10)
		(v6) edge (v11)
		(v4) edge (v9)
		(v1) edge (v12)
		(v8) edge (v13)
		(v15) edge (v16)
		(v15) edge (v28)
		(v16) edge (v17)
		(v17) edge (v22)
		(v18) edge (v27)
		(v19) edge (v20)
		(v20) edge (v21)
		(v21) edge (v22)
		(v23) edge (v24)
		(v23) edge (v28)
		(v24) edge (v25)
		(v25) edge (v26)
		(v26) edge (v27)
		(v21) edge (v28)
		(v16) edge (v19)
		(v17) edge (v24)
		(v20) edge (v25)
		(v18) edge (v23)
		(v15) edge (v26)
		(v22) edge (v27)
		(v29) edge (v30)
		(v29) edge (v42)
		(v31) edge (v36)
		(v32) edge (v33)
		(v32) edge (v41)
		(v33) edge (v34)
		(v34) edge (v35)
		(v35) edge (v36)
		(v37) edge (v38)
		(v37) edge (v42)
		(v38) edge (v39)
		(v39) edge (v40)
		(v40) edge (v41)
		(v35) edge (v42)
		(v30) edge (v33)
		(v31) edge (v38)
		(v34) edge (v39)
		(v32) edge (v37)
		(v29) edge (v40)
		(v36) edge (v41)
		(v11) edge (v44)
		(v10) edge (v47)
		(v31) edge (v49)
		(v30) edge (v48)
		(v19) edge (v43)
		(v18) edge (v46)
		(v47) edge (v50)
		(v46) edge (v50)
		(v43) edge (v45)
		(v44) edge (v45)
		(v48) edge (v50)
		(v45) edge (v49)
		(v46) edge (v49)
		(v43) edge (v47)
		(v44) edge (v48)
		;
	\end{tikzpicture}
	\caption{The Levi graph of a $25_3$ configuration with no blocking set}
	\label{fig:BS25}
\end{figure}
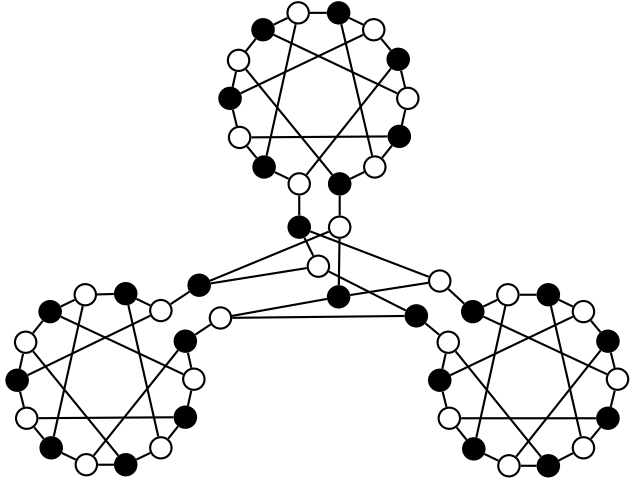

In fact this graph has already appeared in the literature. It appears as Figure 8 in~\cite{mccuaig2000even} as a 2-connected unbalanced 1-extendible cubic bipartite graph. We have the following result.

\begin{theorem}\label{thm:2conn}
A 2-connected symmetric configuration $v_3$ without a blocking set exists if and only if $v\in\{21,22,25\}$ or $v\geq 27$.
\end{theorem}

\subsection{Enumeration of configurations}\label{sec:enum}
Finally in this section we present some enumeration results.  As stated above, for $7 \leq v \leq 18$ there exist just two symmetric configurations with no blocking set; unique $7_3$ and $13_3$ systems. In~\cite{Gropp1997}, Gropp reported that there exist at least four configurations $19_3$ without a blocking set. Recently the present authors~\cite{EGS2019} have enumerated all configurations $19_3$ and we confirm that there are exactly four without a blocking set. These have a nice description as follows. Because $19=13+7-1$, it must be true that at least some of the four configurations $19_3$ can be obtained by using Theorem~\ref{thm:stitch2} with the unique $13_3$ and $7_3$ configurations without blocking sets. We may use the construction of Theorem~\ref{thm:stitch2} with $v=7,v'=13$, taking all possible choices for the distinguished point and block in the two constituent configurations. To this set of configurations we may add those obtained by taking $v=13,v'=7$ in the same way. Finally, this set of configurations can be reduced to isomorphism class representatives using the \texttt{GAP} package \texttt{DESIGN}~\cite{GAP4,DESIGN}. In this way we were able to determine that the method results in exactly four isomorphism classes of configurations $19_3$ with no blocking set. Thus these correspond precisely to the four in the enumeration; this description of the four $19_3$ configurations was known to Gropp and the construction is described in~\cite{DiPaola1991,Gropp1997}. The blocks of these four $19_3$ systems are as follows:

\begin{config}
\-\ 012 034 056 135 19a 236 245 4bc 678 79b 7ac 89c 8de afi bgh dfh dgi efg ehi

\-\ 012 034 056 135 146 29a 2bc 367 4gh 5fi 79b 7ac 89c 8ab 8de dfh dgi efg ehi

\-\ 012 034 056 135 146 236 2de 4fi 5gh 79b 7ac 7fh 89c 8ab 8gi 9ad bcd efg ehi

\-\ 012 034 056 135 146 29a 2bc 367 458 79b 7ac 89c 8de afi bgh dfh dgi efg ehi
\end{config}

Although there is no symmetric configuration $20_3$ without a blocking set, increases in computer power allowed us to extend the enumeration of symmetric configurations to the case where  $v=20$ and this information is summarised in Table~\ref{tab:enum}. Our enumeration, in common with our previous results~\cite{EGS2019}, was carried out using the program \texttt{confibaum} as used in~\cite{Betten2000}. We are grateful to G. Brinkmann for this program and for assistance in our previous enumeration.

The enumeration confirms the fact that there is no symmetric configuration $20_3$ without a blocking set. For completeness we describe here the properties enumerated in Table~\ref{tab:enum}, following the notation of~\cite{Betten2000}. For a configuration $\mathcal{X}$, an \emph{automorphism} is a permutation of the points and blocks of $\mathcal{X}$ which preserves incidence. The \emph{dual} of $\mathcal{X}$ is the configuration obtained by reversing the roles of the points and blocks of $\mathcal{X}$. If $\mathcal{X}$ is isomorphic to its dual, we say it is \emph{self-dual}, and an isomorphism between $\mathcal{X}$ and its dual is an \emph{anti-automorphism}. An anti-automorphism of $\mathcal{X}$ of order two is called a \emph{polarity}, and a configuration admitting such an isomorphism is \emph{self-polar}. The group of all automorphisms of $\mathcal{X}$ (preserving the roles of points and blocks) is denoted by $\Aut(\mathcal{X})$, and the group of all automorphisms and anti-automorphisms by $A(\mathcal{X})$. If $\Aut(\mathcal{X})$ acts transitively on the points of $\mathcal{X}$ then we say $\mathcal{X}$ is \emph{point-transitive}. A \emph{flag} of $\mathcal{X}$ is an ordered pair $(p,B)$ with $p\in B$; if $\Aut(\mathcal{X})$ acts transitively on the set of flags then we say $\mathcal{X}$ is \emph{flag-transitive}; if $A(\mathcal{X})$ acts transitively on the set of flags regarded as \emph{unordered} pairs, then we say $\mathcal{X}$ is \emph{weakly flag-transitive}. A \emph{cyclic} configuration $\mathcal{X}$ is one admitting a cyclic subgroup of $\Aut(\mathcal{X})$ acting regularly on points. 
\begin{table}[h]
	\centering\setlength{\tabcolsep}{1em}
	\begin{tabular}{rrrrrrrrrr}
		\hline
		$v$ & $a$ & $b$ & $c$ & $d$ & $e$ & $f$ & $g$ & $h$ & $i$ \\
		\hline
		7 & 1 & 1 & 1 & 1 & 1 & 1 & 1 & 1 & 0 \\ 
		8 & 1 & 1 & 1 & 1 & 1 & 1 & 1 & 0 & 0 \\ 
		9 & 3 & 3 & 3 & 2 & 1 & 1 & 1 & 0 & 0 \\ 
		10 & 10 & 10 & 10 & 2 & 1 & 1 & 1 & 0 & 0 \\ 
		11 & 31 & 25 & 25 & 1 & 1 & 0 & 0 & 0 & 0 \\ 
		12 & 229 & 95 & 95 & 4 & 3 & 1 & 1 & 0 & 0 \\ 
		13 & 2,036 & 366 & 365 & 2 & 2 & 1 & 1 & 1 & 0 \\ 
		14 & 21,399 & 1,433 & 1,432 & 3 & 3 & 1 & 1 & 0 & 1 \\ 
		15 & 245,342 & 5,802 & 5,799 & 5 & 4 & 1 & 1 & 0 & 1 \\ 
		16 & 3,004,881 & 24,105 & 24,092 & 6 & 4 & 2 & 2 & 0 & 4 \\ 
		17 & 38,904,499 & 102,479 & 102,413 & 2 & 2 & 0 & 0 & 0 & 13 \\ 
		18 & 530,452,205 & 445,577 & 445,363 & 9 & 5 & 1 & 1 & 0 & 47 \\ 
		19 & 7,597,040,188 & 1,979,772 & 1,979,048 & 3 & 3 & 1 & 1 & 4 & 290 \\ 
		20 & 114,069,332,027 & 8,981,097 & 8,978,373 & 9 & 5 & 2 & 2 & 0 & 2,413 \\ 
		\hline
		\multicolumn{10}{p{15cm}}{\footnotesize
			\textbf{Note}: $a$ is the number of configurations $v_3$; $b$ is the number of self-dual configurations; $c$ is the number of self-polar configurations; $d$ is the number of point-transitive configurations; $e$ is the number of cyclic configurations; $f$ is the number of flag-transitive configurations; $g$ is the number of weakly flag-transitive configurations; $h$ is the number of connected blocking set-free configurations; $i$ is the number of disconnected configurations.
		} 
	\end{tabular}
	\caption{Numbers of configurations $v_3$}
	\label{tab:enum}
\end{table}

Note that for consistency with previously published results, the counts in Table~\ref{tab:enum} include disconnected configurations.

The next case to consider is $v = 21$. A $21_3$ configuration without a blocking set can be constructed from three $7_3$ configurations by Theorem~\ref{thm:stitch3}.  Because the automorphism group of the Fano plane is flag-transitive, all systems constructed by this method are isomorphic.  We show that this is the unique system of this order without a blocking set. From Theorem~\ref{thm:3conn}, any such system is 2-connected. 

The first observation to make is that a cubic bipartite graph with edge connectivity 2 and edge cutset $\{ab,cd\}$ must take the form illustrated in Figure~\ref{fig:conn2}. In the diagram, the circles represent the components $C_1,C_2$ following the edge cut and the black/white colouring of the vertices represents the bipartition of the graph.

\begin{figure}\centering
	\begin{tikzpicture}[
		x=1cm,
		y=-1cm,
		inner sep=0.1mm,
		scale=0.6,
		thick,
		point/.style={circle,draw,minimum size=6,fill=black},
		block/.style={circle,draw,minimum size=6,fill=white},
		edge label/.style={fill=white}
		]
		\node at (5,2) [point,label={$a$}] (v1) {};
		\node at (5,4) [block,label={$d$}] (v2) {};
		\node at (8,2) [block,label={$b$}] (v3) {};
		\node at (8,4) [point,label={$c$}] (v4) {};
		\node at (3,3) [label={$C_1$}] {};
		\node at (10,3) [label={$C_2$}] {};
		
		\draw (3,3) circle(3) [dotted];
		\draw (10,3) circle(3) [dotted];
		
		\path
		(v1) edge (v3)
		(v2) edge (v4)
		;
	\end{tikzpicture}
	\caption{A cubic bipartite graph with edge connectivity 2}
	\label{fig:conn2}
\end{figure}
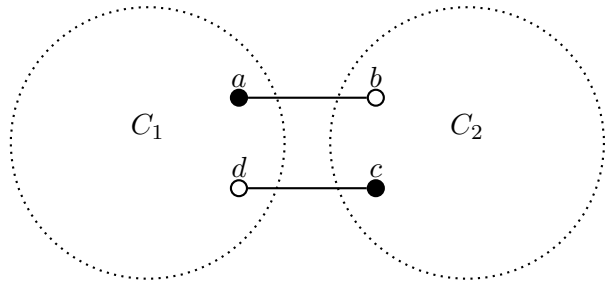

Suppose now that the graph in Figure~\ref{fig:conn2} is the Levi graph of a symmetric configuration $21_3$. Say the components $C_1,C_2$ following the edge cut have respective orders $n_1$ and $n_2$, with $n_1+n_2=42$. Then $C_1$ has $n_1-2$ vertices of valency 3, and 2 vertices ($a$ and $d$) of valency 2. In other words, it is a subcubic bipartite graph with $3n_1/2-1$ edges. A similar argument holds for $C_2$, where the distinguished vertices of valency 2 are $b$ and $c$.

The problem of constructing all cubic bipartite graphs with edge connectivity 2 can therefore be reduced to finding all possible components $C_1,C_2$. Note that a component is not necessarily an edge-deleted Levi graph of some configuration; this will be the case for $C_1$ for example if and only if the distance between the distinguished vertices $a$ and $d$ is at least 5. But these vertices may be at distance 3 or even 1. However the component can contain no 4-cycles. By using the \texttt{genbg} utility provided in the \texttt{nauty} package~\cite{mckay201494}, we may use a computer to construct all possible components. This computer search shows that the smallest possible one of these has order 14 and is unique; it is an edge-deleted Heawood graph. At order 16 there are three possible components: one with $a,d$ at distance 5 which is an edge-deleted Levi graph of the $8_3$ configuration; one with $a,d$ at distance 3 and one with $a,d$ adjacent.

In principle then, all cubic bipartite graphs with edge connectivity 2, girth at least 6 and order $42$ can be constructed by finding all possible components $C_1,C_2$ such that $n_1+n_2=42$ and joining their distinguished vertices as in Figure~\ref{fig:conn2}. The join can be done in two (possibly) non-isomorphic ways and is subject to the constraint that at least one of $C_1,C_2$ must have its distinguished vertices non-adjacent (to avoid creating a 4-cycle).

We therefore proceed as follows. For $n=14,16,\ldots,28$ we generate using \texttt{genbg} all subcubic bipartite graphs of order $n$ and girth at least 6 with $3n/2-1$ edges. Then the idea is that we connect up a graph of order $n$ with a graph of order $42-n$ as above, subject to the constraint noted. The resulting cubic graph will have girth at least 6; this process therefore generates the entire population of 2-edge-connected cubic bipartite graphs of order 42. Any blocking set free configuration at $v=21$ must have a Levi graph within this population.

Although there are a large number of possible components, it turns out that with modern computers the generation of the components and hence the enumeration of all possible Levi graphs of configurations $21_3$ could be completed. Exactly one of the resulting Levi graphs arose from a configuration which failed to have a blocking set. It is illustrated in Figure~\ref{fig:BS22} and its blocks are as follows.

\begin{config}
\-\ 012 034 056 135 146 236 24l 58f 79c 7ak 7bl 89a 8ck 9bk abc deh dfj dgi egj eil fgh hij
\end{config}

We therefore have the following result.

\begin{theorem}\label{thm:bsfree21}
	There is a unique symmetric configuration $21_3$ having no blocking set; it is the configuration obtained by using three Fano planes in the construction of Theorem~\ref{thm:stitch3}.
\end{theorem}

As noted above, the symmetric configuration $22_3$ with no blocking set illustrated in Figure~\ref{fig:BS22} was found by Dorwart and Gr\"unbaum~\cite{Dorwart1992}. In fact we can show that this also is the unique such configuration. We use the same procedure as for the $21_3$ configuration, but the search can be considerably shortened by the following simple lemma.
\begin{lemma}\label{lem:ham}
	Let $v\geq 8$ be an even number. If a symmetric configuration $v_3$ contains no blocking set, then its Levi graph is non-Hamiltonian.
\end{lemma}
\begin{proof}
	Suppose that the Levi graph contains a Hamiltonian cycle $p_0 B_0 p_1 B_1 \ldots p_{v-1} B_{v-1} p_0$ where the $p_i$ and $B_i$ respectively represent point and block vertices, $i=0,1,\ldots, v-1$. Colour the even numbered points $p_0,p_2,\ldots,p_{v-2}$ red and the odd numbered points blue. Since $v$ is even, no block is monochromatic and so the even numbered points form a blocking set for the configuration.
\end{proof}
Lemma~\ref{lem:ham} and Theorem~\ref{thm:3conn} show that if a symmetric configuration $22_3$ has no blocking set, its Levi graph must be a 2-connected non-Hamiltonian cubic bipartite graph of order 44. Thus an enumeration of blocking set free configurations on 22 points can be achieved by an exhaustive enumeration of such graphs.

We use the same basic search methodology as in the $21_3$ case, but this time we extend the generation of the components $C_1,C_2$ up to order 30. To guarantee that the resulting graph of order 44 will be non-Hamiltonian, we require that at least one of $C_1,C_2$ must fail to have a Hamiltonian path between its distinguished vertices. (To check the existence of a Hamiltonian path, we create an augmented graph in which the two distinguished vertices of valency 2 are joined to a new vertex; then the augmented graph is Hamiltonian if and only if there is a Hamiltonian path between the distinguished vertices in the original graph. This technique allows us to use the well tested \texttt{cubhamg} utility in the \texttt{nauty} package, rather than writing new software for the Hamiltonicity test.)

Restricting our search to pairs $C_1,C_2$ such that at least one component fails to have a Hamiltonian path between the distinguished vertices gives a very substantial reduction in the number of component pairs to be considered. We were thus able to complete the enumeration of the 2-connected non-Hamiltonian cubic bipartite graphs of order 44, and found that only one of these is the Levi graph of a blocking set free configuration $22_3$. Thus we have the following result.

\begin{theorem}\label{thm:bsfree22}
	There is a unique symmetric configuration $22_3$ having no blocking set; it is the configuration of Dorwart and Gr\"unbaum~\cite{Dorwart1992}.
\end{theorem}

Next, a $25_3$ configuration without a blocking set can be constructed by Theorem~\ref{thm:stitch2} using either two $13_3$s or a $7_3$ with one of the $19_3$s. Again in~\cite{Gropp1997}, Gropp reports that there are at least 19 such configurations. With the assistance of computers in a similar way to the construction of the $19_3$ configurations, in fact we find 23 isomorphism classes of configurations $25_3$ arising from Theorem~\ref{thm:stitch2} in this way. The blocks of these are given in the Appendix.

All of these systems have connectivity 3 and we now know that there is at least one further system which is 2-connected; thus 25 is the smallest order for which there exist both 3-connected and 2-connected blocking set free systems. Using Theorem~\ref{thm:3conn} we can now make an enumeration of 3-connected symmetric configurations $v_3$ without a blocking set for $v\in\{7,13,19,25,31,37,43\}$. We do this by repeated application of the $v+v'-1$ construction in all possible ways, and reducing the resulting configurations to a set of isomorphism class representatives. The results are shown in Table~\ref{tab:conn3enum}.

\begin{table}\centering
\begin{tabular}{crrr}
\hline
$v$ & Configurations & Self-dual & Self-polar \\
\hline
7 & 1 & 1 & 1 \\
13 & 1 & 1 & 1 \\
19 & 4 & 2 & 2 \\
25 & 23 & 5 & 5 \\
31 & 182 & 14 & 14 \\
37 & 1,747 & 45 & 45 \\
43 & 19,485 & 145 & 145 \\
\hline
\end{tabular}
\caption{Numbers of 3-connected blocking set free configurations $v_3$}
\label{tab:conn3enum}
\end{table}

\section{Strong colourings}\label{sec:strong}
In this section we turn our attention to the strong chromatic number $\chi_s$ of a symmetric configuration, and also investigate its relationship to the weak chromatic number $\chi_w$. Our first observation is that the strong chromatic number of a configuration is equal to the chromatic number of its associated graph. Since the associated graph is regular of valency 6 and contains triangles, it follows from Brooks' Theorem that $\chi_s\in\{3,4,5,6,7\}$, and $\chi_s=7$ if and only if the associated graph is a complete graph; that is to say, the configuration is the Fano plane.

The first case to consider is $\chi_s=3$. An immediate observation is that each block of the configuration must contain exactly one point from each of the three colour classes, and so $v\equiv 0\!\pmod{3}$. By colouring two classes in the strong colouring (say) red and the third blue, we see that $\chi_s=3$ implies $\chi_w=2$. 

A nice description of strongly 3-chromatic configurations is as follows. From the associated graph of the configuration, form the subgraph induced by the points from any two of the three colour classes. It is easy to see that this induced subgraph is a cubic bipartite graph (not necessarily connected), and a given strong 3-colouring of a configuration gives rise to three cubic bipartite graphs in this way by deleting each of the colour classes. Given any cubic bipartite graph $\Gamma$, it is natural to ask whether $\Gamma$ can arise in this way. Our next result answers this in the affirmative.

\begin{theorem}\label{thm:delgraph}
Let $m\geq 3$ and let $\Gamma$ be a cubic bipartite graph of order $2m$. Then there exists a strongly 3-chromatic symmetric configuration $\mathcal{X}$ on $3m$ points, and a strong 3-colouring of $\mathcal{X}$, such that the induced subgraph of the associated graph of $\mathcal{X}$ formed by deleting the points of one colour class is isomorphic to $\Gamma$.
\end{theorem}
\begin{proof}
Our aim is to construct a new 6-regular graph $\Gamma'$ on $3m$ vertices to be the associated graph of our configuration $\mathcal{X}$. We begin by creating three sets of vertices $V_1,V_2,V_3$, each of order $m$. Between the vertices of $V_1$ and $V_2$ we add edges such that the induced subgraph on $V_1\cup V_2$ is isomorphic to $\Gamma$. We now note that by~\cite{Bialostocki1982}, the edges of $\Gamma$ can be decomposed into a collection of $m$ copies of the graph $3K_2$, i.e. a collection of $m$ sets of three disjoint edges. Each of the $m$ sets of three edges contains exactly six vertices; we construct $\Gamma'$ by joining each of the $m$ vertices in $V_3$ to all the vertices in exactly one of these sets.

Since $\Gamma'$ is a 6-regular tripartite graph, any decomposition of its edge set into triangles will yield a strongly 3-chromatic configuration on $3m$ points, where the colour classes are the sets $V_1,V_2,V_3$. A suitable triangle decomposition is given by using each edge between vertices in $V_1$ and $V_2$ together with the two edges joining its endpoints to a vertex in $V_3$. By construction, the configuration $\mathcal{X}$ represented by this decomposition has the required properties, taking the colour class assigned to $V_3$ as the one to be deleted.
\end{proof}

In general, the three cubic bipartite graphs formed by deleting a colour class from a strongly 3-chromatic configuration in this way will not be isomorphic. Another natural question is whether we can construct strongly 3-chromatic configurations in such a way that, with a suitable colouring, the resulting colour class deleted graphs are actually isomorphic. It turns out that we can do this for any $v\geq 9$ which is a multiple of 3.

\begin{theorem}\label{thm:col3iso}
Let $s\geq 3$ and let $v=3s$. Then there is a cubic bipartite graph $\Gamma$ of order $2s$, and a strongly 3-chromatic symmetric configuration $\mathcal{X}$ on $v$ points, such that deleting any of the three colour classes in a suitable colouring of $\mathcal{X}$ we obtain a graph isomorphic to $\Gamma$.
\end{theorem}
\begin{proof}
We begin by defining a suitable cubic bipartite graph $\Gamma$. Let the vertex set of $\Gamma$ consist of $\{a_0,a_1,\ldots,a_{s-1}\}\cup\{b_0,b_1,\ldots,b_{s-1}\}$. There is an edge from $a_i$ to $b_j$ if and only if $i-j\in\{-1,0,1\}$, where of course the arithmetic is modulo $s$. Now we extend $\Gamma$ to a 6-regular graph $\Gamma'$, and colour the edges in a particular way. To create $\Gamma'$, create a new vertex set $\{c_0,c_1,\ldots,c_{s-1}\}$ and join edges $c_i$ to $a_j$ and $c_i$ to $b_j$ exactly as in $\Gamma$. A triangle decomposition in $\Gamma'$ can be defined as follows. For each $i=0,1,\ldots,s-1$, colour the edges in $\Gamma'$ according to the following rules:
\begin{itemize}[itemsep=0pt]
	\item Edges from $a_i$ to $b_{i-i}$, $b_i$ to $c_i$ and $c_i$ to $a_{i+1}$ are coloured red.
	\item Edges from $a_i$ to $b_i$, $b_i$ to $c_{i+1}$ and $c_i$ to $a_{i-1}$ are coloured green.
	\item Edges from $a_i$ to $b_{i+1}$, $b_i$ to $c_{i-1}$ and $c_i$ to $a_i$ are coloured blue.
\end{itemize}
Then the monochromatic triangles in the above edge colouring form a triangle decomposition of $\Gamma'$. The configuration $\mathcal{X}$ represented by this decomposition is strongly 3-chromatic (since $\Gamma'$ is tripartite) and deleting any of the three sets in the tripartition leaves a graph isomorphic to $\Gamma$.
\end{proof}
Note that the symmetric configuration constructed in the above theorem is resolvable, the sets of monochromatic triangles of the three colours forming the resolution classes. The graph $\Gamma'$ is a Cayley graph of the group $\mathbb{Z}_3\times\mathbb{Z}_s$.

We next turn our attention to the case $\chi_s=4$. It is easy to see that a strongly 4-chromatic configuration is weakly 2-chromatic; if we strongly colour the configuration with colours 1,2,3,4 then we can colour the points in colour classes 1 and 2 blue, and the remainder red. Then no block is monochromatic. 

In Table~\ref{tab:chis} we give computer calculations of the strong chromatic numbers of all connected configurations with $v\leq 15$; the numerical evidence is that the case $\chi_s=4$ seems to be the most common. Indeed, our next result shows that we can construct a symmetric configuration with $\chi_s=4$ for all $v\geq 8$
.
\begin{table}[h]\centering
	\begin{tabular}{c*{6}{p{1.5cm}}}
		\hline
		$v$ & Total & $\chi_s=3$ & $\chi_s=4$ & $\chi_s=5$ & $\chi_s=6$ & $\chi_s=7$ \\
		\hline
		7 & 1 & 0 & 0 & 0 &  0 & 1 \\
		8 & 1 & 0 & 1 & 0 &  0 & 0 \\
		9 & 3 & 1 & 1 & 1 & 0  & 0 \\
		10 & 10 & 0 & 3 & 7 & 0  & 0 \\
		11 & 31 & 0 & 21 & 9 & 1  & 0 \\
		12 & 229 & 4 & 161 & 64 & 0  & 0 \\
		13 & 2,036 & 0 & 1,451 & 584 & 1  & 0 \\
		14 & 21,398 & 0 & 17,342 & 4,053 & 3 &  0 \\
		15 & 245,341 & 251 & 234,139 & 10,938 & 13 &  0 \\
		\hline
	\end{tabular}
	\caption{Strong chromatic numbers of connected configurations $v_3$}
	\label{tab:chis}
\end{table}

\begin{theorem}\label{thm:chis4}
There exists a strongly 4-chromatic configuration $v_3$ for all $v\geq 8$.
\end{theorem}
\begin{proof}
The proof is similar to that of Theorem~\ref{thm:minblocking}. We again use Martinetti's extension operation, though the replacement of blocks is different from that done in Theorem~\ref{thm:minblocking}. First observe from Table~\ref{tab:chis} that a strongly 4-chromatic configuration $v_3$ exists for $8\leq v\leq 12$.

Let $v=3s$ where $s\geq 4$, and let $V=\{a_i,b_i,c_i:0\leq i\leq s-1\}$. Let the blocks of the symmetric configuration $v_3$ be the sets $\{a_i,b_i,c_{i+1}\}$, $\{a_i,b_{i+1},c_i\}$ and $\{a_{i+1},b_i,c_i\}$, $0\leq i\leq s-1$.

Now suppose that $v\equiv 1\!\pmod{3}$, $v\geq 13$. Construct a configuration $(v-1)_3$ as above. Introduce a new point $\infty_0$ and use the extension operation, replacing the blocks $\{a_0,b_0,c_1\}$ and $\{a_1,b_1,c_2\}$ by blocks $\{\infty_0,b_0,c_2\}$, $\{\infty_0,a_0,c_1\}$ and $\{\infty_0,a_1,b_1\}$.

Next suppose that $v\equiv 2\!\pmod{3}$, $v\geq 14$. Construct a configuration $(v-1)_3$ as above. Introduce a new point $\infty_1$ and again use the extension operation, replacing the blocks $\{a_0,b_1,c_0\}$ and $\{a_1,b_2,c_1\}$ by blocks $\{\infty_1,a_0,b_2\}$, $\{\infty_1,b_1,c_0\}$ and $\{\infty_1,a_1,c_1\}$.

Finally, suppose that $v\equiv 0\!\pmod{3}$, $v\geq 15$. Construct a configuration $(v-1)_3$ as above. Introduce a new point $\infty_2$ and again use the extension operation, replacing the blocks $\{a_1,b_0,c_0\}$ and $\{a_2,b_1,c_1\}$ by blocks $\{\infty_2,a_2,c_0\}$, $\{\infty_2,a_1,b_0\}$ and $\{\infty_2,b_1,c_1\}$. 

In all cases it is clear that the symmetric configurations so constructed have a strong colouring with 4 colours and therefore in the first two cases are strongly 4-chromatic. It remains to prove that in the case where $v\equiv 0\!\pmod 3$ it is not 3-chromatic. Suppose that it is and that in the block $\{\infty_0,b_0,c_2\}$, $\infty_0$ receives colour 1, $b_0$ receives colour 2 and $c_0$ receives colour 3. Then in the block $\{\infty_0,a_0,c_1\}$, $a_0$ and $c_1$ receive colours 2 and 3 in some order and likewise in the block $\{\infty_0,a_1,b_1\}$, $a_1$ and $b_1$ receive colours 2 and 3 in some order, giving four possibilities in all. However in all cases either $a_1$ and $c_1$ or $b_1$ and $c_1$ receive the same colour, giving a contradiction.
\end{proof}

Before considering the next case $\chi_s=5$, we state and prove the following theorem which gives the strong chromatic number of certain cyclic configurations.

\newpage
\begin{theorem}\label{thm:chis45}
Let $v\geq 7$ and let $C_v$ be the cyclic configuration on $v$ points generated by the block $\{0,1,3\}$ under the mapping $i\mapsto i+1\!\pmod{v}$. Then
\[\chi_s(C_v)=
\begin{cases}
7 & \text{if } v=7;\\
6 & \text{if } v=11;\\
5 & \text{if } v\equiv 1,2,3\!\pmod{4},v\notin \{7,11\};\\
4 & \text{if } v\equiv 0\!\pmod{4}.\\
\end{cases}
\]
\end{theorem}
Proof of this theorem is facilitated by the following lemma.
\begin{lemma}\label{lem:chis}
Let $v\geq 7$ and let $C_v$ be the cyclic configuration on $v$ points generated by the block $\{0,1,3\}$. If $v$ can be expressed in the form $4a+5b$ where $a,b$ are non-negative integers, then $\chi_s(C_v)\leq 5$.
\end{lemma}
\begin{proof}
We consider the points of $C_v$ to be elements of the cyclic group $\Z_v$. The colours of the points will be taken from the set $\{0,1,2,3,4\}$. Each point $i$ is assigned colour $c(i)$ as follows.
\[
c(i)=
\begin{cases}
i\!\!\mod 4 & \text{if }i<4a;\\
(i-4a)\!\!\mod 5 & \text{if }i\geq 4a.
\end{cases}
\]
So listing the elements of $\Z_v$ from $0$ to $v-1$ in order, the assignment of colours looks like:
\[0123\ 0123\ldots 0123\ 01234\ 01234\ldots 01234.\]
It is easy to see that if $b>0$ this represents a strong 5-colouring of $C_v$; and if $v$ is divisible by 4 we can write $v=4a$ and it is a 4-colouring.
\end{proof}
\begin{proof}[Proof of Theorem~\ref{thm:chis45}]
The block $\{0,1,3\}$ shows that 0, 1 and 3 must be assigned different colours; then the blocks $\{1,2,4\}$, $\{2,3,5\}$ and $\{v-1,0,2\}$ show that 2 must be assigned a fourth colour. So for any $v\geq 7$, $\chi_s(C_v)\geq 4$.

If $v$ can be written in the form $4a+5b$, then Lemma~\ref{lem:chis} applies and so $\chi_s(C_v)$ will equal 4 if $v\equiv 0\!\pmod{4}$. If $v\not\equiv 0\!\pmod{4}$, then by the paragraph above and the proof of Lemma~\ref{lem:chis}, the assignment of colours in a strong 4-colouring would have to be $0123\ 0123\ldots 0123$ which is impossible because the points cannot be split into groups of 4. So $\chi_s(C_v)=5$.

The only values of $v\geq 7$ which cannot be written in the form $4a+5b$ are 7 and 11. If $v=7$ then the associated graph of $C_v$ is the complete graph $K_7$ and this has chromatic number 7. If $v=11$ then Lemma~\ref{lem:chis} cannot be applied, and computer testing shows that $\chi_s(C_{11})=6$. In fact as Table~\ref{tab:chis} shows, this is the unique 6-chromatic configuration $11_3$.
\end{proof}

The case $\chi_s=5$ is interesting. Theorem~\ref{thm:chis45} shows that symmetric configurations $v_3$ with $\chi_s=5$ exist for all $v\equiv 1,2,3\!\pmod{4}$, $v\notin\{7,11\}$ and Table~\ref{tab:chis} shows that such a configuration also exists for $v=11$ but not $v=7$. It remains to determine existence for $v\equiv 0\!\pmod{4}$, which it is more appropriate for us to do later in Theorem~\ref{thm:chis5}.

All examples of strongly 5-chromatic configurations $v_3$ with $v\leq 15$ have $\chi_w=2$, and indeed all other examples we have seen have $\chi_w=2$ (in other words, the configuration contains a blocking set). However, we have been unable to find a proof of this, and so the existence of a symmetric configuration with $\chi_s=5$ and $\chi_w=3$ remains an open question. As a partial result in this direction, we can show that all configurations which are ``almost'' strongly 4-colourable have weak chromatic number 2.

\begin{theorem}\label{thm:4col}
Suppose that we have a strongly 5-chromatic configuration $v_3$ in which all but at most 2 points can be coloured using 4 colours. Then the weak chromatic number of the configuration is 2.
\end{theorem}
\begin{proof}
Firstly, note that in any 5-colouring each of the $v$ blocks is coloured with one of the $\binom{5}{3}=10$ possible sets of 3 colours; and each of these sets must appear at least once if the weak chromatic number is 3. (If a set of 3 colours does not appear in any block, we can assign blue to these 3 and red to the other 2 to get a weak 2-colouring.)

Now suppose that we can assign 4 colours (say 1, 2, 3 and 4) to $v-1$ points so that no colour is repeated in a block. Clearly we can assign a 5th colour 5 to the remaining point, and in this 5-colouring at least 3 sets of 3 colours must fail to appear in any block, since there are 6 possible sets containing this colour but only 3 blocks containing the single vertex with this colour. Thus the configuration has weak chromatic number 2.

If we can only assign 4 colours to $v-2$ points the position is more awkward. Let the two uncoloured points be $a$ and $b$. Suppose first that $a$ and $b$ do not appear in the same block. We seek an assignment of two of the existing colours 1, 2, 3, 4 to red and the remaining two to blue, such that we can choose red or blue for $a$ and $b$ to obtain a weak 2-colouring. There are exactly three ways to do this initial red/blue assignment: 12/34, 13/24 and 14/23. We shall call an assignment \emph{compatible} with $a$ if it leaves a possible red/blue choice for $a$ such that no monochromatic block is created. Since $a$ appears in three blocks, it is easy to see that at most one of the three possible assignments is not compatible with $a$. For example, if the colours of the other points in the blocks containing $a$ are $\{1,2\}$, $\{3,4\}$ and $\{1,4\}$, then the assignment 12/34 is incompatible with $a$ but the assignments 13/24 and 14/23 are compatible. Since the same argument holds for $b$, at least one possible assignment is compatible with $a$ and $b$ and so the configuration has a weak 2-colouring.

If $a$ and $b$ do appear in the same block, then each has two other blocks in which it appears. In this case, not only is there an assignment compatible with both $a$ and $b$, but the choice of red/blue for $a$ and $b$ may be made freely. So we can choose red for $a$ and blue for $b$ and again there is a weak 2-colouring.
\end{proof}

Now we come to the case $\chi_s=6$. Table~\ref{tab:chis} shows that this is uncommon; of the 269,049 connected configurations $v_3$ with $8\leq v\leq 15$, only 18 are strongly 6-chromatic. These are given in the Appendix. Nevertheless, we are able to deduce the existence of strongly 6-chromatic configurations for almost all values of $v$ as the next result shows.

\begin{theorem}\label{thm:chis6}
There exists a strongly 6-chromatic connected configuration $v_3$ for $v=11$ and for all $v\geq 13$.
\end{theorem}
\begin{proof}
The cases $v=11$ and $v=13$ follow from Table~\ref{tab:chis}. So let $v\geq 14$ and let $C_7$ be the cyclic configuration on 7 points generated by the block $\{0,1,3\}$ under the mapping $i\mapsto i+1\!\pmod{7}$; this is of course the unique $7_3$ configuration and is strongly 7-chromatic. Now choose any connected configuration $(v-7)_3$ and number the points from $7$ to $v-1$. By relabelling if necessary, we may assume without loss of generality that this configuration contains the block $\{7,8,9\}$. Now create a new configuration $\mathcal{X}$ with the blocks of these two configurations, but replacing the blocks $\{0,1,3\}$ and $\{7,8,9\}$ with $\{1,3,7\}$ and $\{0,8,9\}$. Suppose $\mathcal{X}$ can be strongly coloured with 5 colours. Then the colours assigned to points 1 to 6 together with a sixth colour for point 0 would give a strong 6-colouring for the original configuration $C_7$, which is impossible. Thus $\chi_s(\mathcal{X})\geq 6$ and since $\chi_s(\mathcal{X})\leq 6$ by Brooks' Theorem, $\mathcal{X}$ is a strongly 6-chromatic connected configuration on $v$ points as required.
\end{proof}

As noted above, all the blocking set free configurations of which we are aware have $\chi_s=6$. However, only the single example at $v=13$ in Table~\ref{tab:chis} has $\chi_s=6$ and $\chi_w=3$, so there are many examples with $\chi_s=6$ and $\chi_w=2$.

Finally in this section we complete the proof of the existence spectrum for $\chi_s=5$ which we earlier deferred until later. It follows the proof of Theorem~\ref{thm:chis6} but is more intricate.

\newpage
\begin{theorem}\label{thm:chis5}
There exists a strongly 5-chromatic connected configuration $v_3$ for all $v\equiv 0\!\pmod{4}$, $v\geq 12$.
\end{theorem}
\begin{proof}
Examples for $v=12$ and $v=16$ are the following.

\begin{config}
\-\ 012 034 056 135 146 237 289 48a 59b 6ab 78b 79a

\-\ 012 034 056 135 146 236 278 479 57a 89b 8cd 9ef ace adf bcf bde
\end{config}

Now let $v\geq 20$ and let $C_{11}$ be the cyclic configuration on 11 points generated by the block $\{0,1,3\}$ under the mapping $i\mapsto i+1\!\pmod{11}$; from Theorem~\ref{thm:chis45} this is strongly 6-chromatic. Now choose any connected configuration $(v-11)_3$ with $\chi_s=5$ and number the points from 11 to $v-1$. From Theorem~\ref{thm:chis45} this is possible. By relabelling if necessary, we may assume without loss of generality that this configuration contains the block $\{11,12,13\}$ and that in the strong 5-colouring, these points receive colours red, yellow and blue respectively. Now create a new configuration $\mathcal{X}$ with the blocks of these two configurations but replacing the blocks $\{0,1,3\}$ and $\{11,12,13\}$ with $\{1,3,11\}$ and $\{0,12,13\}$. Suppose $\mathcal{X}$ can be strongly coloured with 4 colours. Then the colours assigned to points 1 to 10 together with a fifth colour assigned to point 0 would give a strong 5-colouring of the original configuration $C_{11}$, which is impossible. Thus $\chi_s(\mathcal{X})\geq 5$, and since $\chi_s(\mathcal{X})\leq 6$ by Brooks' Theorem, $\mathcal{X}$ is either strongly 5-chromatic or 6-chromatic. It remains to show that it is the former by exhibiting a colouring.

Colour the blocks of the $(v-11)_3$ configuration without the block $\{11,12,13\}$ with five colours red, yellow, blue, green and white, respecting that colours have already been assigned to points 11, 12 and 13. Colour the remaining points as follows: 4 and 8 red; 2 and 9 yellow; 3 and 7 blue; 0, 1 and 5 green; 6 and 10 white.
\end{proof}
\section{Open questions}\label{sec:open}
We gather here some of the interesting open questions arising from this research. The first of these relates to symmetric configurations $25_3$ without a blocking set. We now have enumerations of all symmetric configurations $v_3$ for $7\leq v\leq 20$ and all 3-connected symmetric configurations without a blocking set for $7\leq v\leq 43$. There are unique configurations $21_3$ and $22_3$ without a blocking set, both necessarily 2-connected, and a 2-connected configuration $25_3$ without a blocking set is known. The question remains whether this is unique.

The second problem is to extend the work on the sizes of minimal blocking sets, possibly along the lines of Theorems~\ref{thm:nearmin} and~\ref{thm:blocking013}. In particular it would be interesting to find constructions of symmetric configurations $v_3$ whose minimal blocking set has maximum cardinality, i.e. $(v-1)/2$ if $v$ is odd and $v/2$ if $v$ is even. The admittedly limited evidence from Table~\ref{tab:bssizes} suggests that such configurations exist except for $v=7$ (where there is no blocking set) and $v=14$, though amongst the set of all symmetric configurations they may be relatively rare. However, given the long history of blocking set free symmetric configurations, finding those with only minimal blocking sets of maximum cardinality may also be quite challenging.

The third problem concerns the relationship between the strong and the weak chromatic numbers. We have observed that if $\chi_s=3$ or $4$ then $\chi_w=2$ and that there are configurations with $(\chi_s,\chi_w)$ equal to both $(6,2)$ and $(6,3)$. However all of the known systems with $\chi_s=5$ have $\chi_w=2$. So we ask does there exist a symmetric configuration $v_3$ with strong chromatic number 5 and weak chromatic number 3? Equivalently, does every blocking set free configuration have strong chromatic number 6?

Finally, as we observed, a given strong 3-colouring of a strongly 3-chromatic configuration gives rise to three cubic bipartite graphs by deleting each of the colour classes from the associated graph of the configuration. Denote these graphs by $\Gamma_1$, $\Gamma_2$ and $\Gamma_3$. In Theorem~\ref{thm:delgraph} we proved that one of these graphs, say $\Gamma_1$, can be any cubic bipartite graph. Now suppose that $\Gamma_1$, $\Gamma_2$ and $\Gamma_3$ are all specified. Does there exist a symmetric configuration $v_3$ whose three cubic bipartite graphs constructed as above are isomorphic to $\Gamma_1$, $\Gamma_2$ and $\Gamma_3$? If not, what are the constraints on these three graphs for this to be possible? The case where $\Gamma_1$, $\Gamma_2$ and $\Gamma_3$ are isomorphic would be of particular interest.

There are of course other problems on symmetric configurations and we hope that this paper will encourage colleagues to work on these.

\section*{Appendix}
The 27 configurations with $v\leq 14$ with a minimal blocking set of size $\lfloor\frac{v}{2}\rfloor$:

\begin{config}
\-\ 012 034 056 135 147 246 257 367

\-\ 012 034 056 135 147 248 267 368 578

\-\ 012 034 056 135 178 247 268 379 469 589

\-\ 012 034 056 137 158 247 268 359 469 789

\-\ 012 034 056 135 146 278 29a 379 47a 589 68a

\-\ 012 034 056 135 147 248 269 37a 59a 68a 789

\-\ 012 034 056 135 147 248 279 36a 59a 689 78a

\-\ 012 034 056 135 147 248 29a 379 58a 67a 689

\-\ 012 034 056 135 147 268 279 389 49a 58a 67a

\-\ 012 034 056 135 178 246 279 37a 49a 58a 689

\-\ 012 034 056 135 148 257 26b 389 49a 67a 79b 8ab

\-\ 012 034 056 135 179 246 278 39a 48a 59b 68b 7ab

\-\ 012 034 056 135 148 257 26c 389 49a 67b 7ac 8ab 9bc

\-\ 012 034 056 135 149 25c 2ab 37b 478 689 6ac 79a 8bc

\-\ 012 034 056 135 167 247 2bc 389 49b 58c 68a 7ab 9ac

\-\ 012 034 056 135 178 239 247 49a 58c 68b 6ac 7ab 9bc

\-\ 012 034 056 135 178 247 289 37b 4ab 59c 69a 6bc 8ac

\-\ 012 034 056 135 178 247 289 3bc 45c 69b 6ac 79a 8ab

\-\ 012 034 056 135 179 247 269 3ac 48a 58c 6bc 78b 9ab

\-\ 012 034 056 135 179 249 26c 38a 4ac 578 68b 7ab 9bc

\-\ 012 034 056 135 179 268 29c 3ab 47a 49b 578 6bc 8ac

\-\ 012 034 056 135 179 26a 289 3ab 478 49b 57c 6bc 8ac

\-\ 012 034 056 135 179 26b 29c 3ab 47a 4bc 578 68c 89a

\-\ 012 034 056 135 17c 24b 26a 38c 469 578 79b 8ab 9ac

\-\ 012 034 056 137 14c 25b 26a 359 468 78c 79a 8ab 9bc

\-\ 012 034 056 137 158 247 2ab 368 49a 59c 6bc 79b 8ac

\-\ 012 034 056 137 189 26b 29c 3ab 45c 49a 578 68a 7bc
\end{config}

The 23 blocking set free configurations $25_3$ arising from Theorem~\ref{thm:stitch2}:

\begin{config}
	\-\ 012 034 056 135 1fg 236 29c 478 4hi 5ab 6de 79b 7ac 89a 8bc dfh dgi efi\\
	\-\ ejk glo hmn jln jmo klm kno
	
	\-\ 012 034 056 135 146 29a 2bc 367 4mn 5lo 79b 7ac 89c 8de 8jk afi bgh dfh \\
	\-\ dgi efg ehi jln jmo klm kno
	
	\-\ 012 034 056 135 146 236 2jk 4lo 5mn 79b 7ac 7ln 89c 8de 8mo 9aj afi bcj \\
	\-\ bgh dfh dgi efg ehi klm kno
	
	\-\ 012 034 056 135 146 236 2de 49c 5ab 79b 7ac 7fg 89a 8bc 8hi dfh dgi efi \\
	\-\ ejk glo hmn jln jmo klm kno
	
	\-\ 012 034 056 135 146 236 2de 49c 5ab 79b 7ac 7fi 89a 8bc 8jk dfg dhi efh \\
	\-\ egi glo hmn jln jmo klm kno
	
	\-\ 012 034 056 135 146 236 28k 4lo 5mn 78j 79b 7ac 89c 9aj afi bcj bgh dfh \\
	\-\ dgi dln efg ehi emo klm kno
	
	\-\ 012 034 056 135 146 236 278 49c 5ab 7ac 7de 89a 8bc 9fg bhi dfh dgi efi \\
	\-\ ejk glo hmn jln jmo klm kno
	
	\-\ 012 034 056 135 146 236 278 49c 5ab 7ac 7de 89a 8bc 9fi bjk dfg dhi efh \\
	\-\ egi glo hmn jln jmo klm kno
	
	\-\ 012 034 056 135 146 236 278 49c 5ab 7ac 7fk 89a 8bc 9ln bmo def dgh dij \\
	\-\ egi ehj fgj hlo imn klm kno
	
	\-\ 012 034 056 135 1de 236 245 4fg 6hi 79b 7ac 7fh 89c 8ab 8gi 9ad bcd efi \\
	\-\ ejk glo hmn jln jmo klm kno
	
	\-\ 012 034 056 135 1de 236 245 4fg 6hi 79b 7ac 7fi 89c 8ab 8jk 9ad bcd efh \\
	\-\ egi glo hmn jln jmo klm kno
	
	\-\ 012 034 056 135 1ab 236 245 4jk 69c 79a 7bc 7de 89b 8ac 8fi dfg dhi efh \\
	\-\ egi glo hmn jln jmo klm kno
	
	\-\ 012 034 056 135 1ef 236 245 4gi 6jk 79b 7ac 7fi 89c 8ab 8lo 9ad bcd deg \\
	\-\ ehi fgh hmn jln jmo klm kno
	
	\-\ 012 034 056 135 1ef 236 245 4gj 6kl 79b 7ac 7kn 89c 8ab 8lm 9ad bcd dho \\
	\-\ egh eij fgi fhj imn kmo lno
	
	\-\ 012 034 056 135 146 236 2dk 4lo 5mn 79b 7ac 7lm 89c 8ab 8no 9ak bck dgh \\
	\-\ dij egi ehj eln fgj fhi fmo
	
	\-\ 012 034 056 135 146 236 27o 4mn 5dl 79b 7ac 89c 8ab 8lm 9ak bck dgh dij \\
	\-\ egi ehj ekn fgj fhi fmo lno
	
	\-\ 012 034 056 135 1bc 236 245 4jk 69a 79b 7ac 7lm 89c 8de 8no afi bgh dfh \\
	\-\ dgi efg ehi jln jmo klo kmn
	
	\-\ 012 034 056 135 1de 236 245 4no 69c 79a 7bc 7jk 89b 8ac 8lm afi bgh dfh \\
	\-\ dgi efg ehi jln jmo klo kmn
	
	\-\ 012 034 056 135 18b 236 245 49c 67a 7bc 7jk 8ac 8lm 9de 9no afi bgh dfh \\
	\-\ dgi efg ehi jln jmo klo kmn
	
	\-\ 012 034 056 135 1fh 236 245 4gi 69e 7ab 7cd 7jk 8ac 8bd 8lm 9ad 9no bfi \\
	\-\ cgh efg ehi jln jmo klo kmn
	
	\-\ 012 034 056 135 146 2bc 2de 369 45a 79b 7ac 7fi 89c 8ab 8jk dfg dhi efh \\
	\-\ egi glo hmn jln jmo klm kno
	
	\-\ 012 034 056 135 146 2bc 2jk 367 458 79b 7ac 89c 8de 9lm afi ano bgh dfh \\
	\-\ dgi efg ehi jln jmo klo kmn
	
	\-\ 012 034 056 135 146 29a 2bc 367 458 79b 7ac 8de 8jk 9lm afi bgh cno dfh \\
	\-\ dgi efg ehi jln jmo klo kmn
\end{config}

The 18 strongly 6-chromatic configurations with $v\leq 15$:

\begin{config}
\-\ 012 034 056 135 147 248 279 36a 59a 689 78a

\-\ 012 034 056 135 146 236 278 49c 5ab 79b 7ac 89a 8bc

\-\ 012 034 056 135 146 236 247 589 7ad 7bc 8ac 8bd 9ab 9cd

\-\ 012 034 056 135 146 28d 29c 368 457 79a 7bc 8ab 9bd acd

\-\ 012 034 056 135 146 29b 2cd 368 457 79a 7bc 89d 8ac abd

\-\ 012 034 056 135 146 28d 29c 36e 457 789 7bc 8ab 9ae acd bde

\-\ 012 034 056 135 147 239 245 6ae 6cd 789 7bc 8ab 8ce 9ad bde

\-\ 012 034 056 135 146 24c 25e 38a 6ce 78d 79e 7ab 89b 9ad bcd

\-\ 012 034 056 135 146 25c 26e 38a 4ce 78d 79e 7ab 89b 9ad bcd

\-\ 012 034 056 135 146 29b 2ae 368 45c 789 7ab 7de 8bd 9ce acd

\-\ 012 034 056 135 146 28a 2de 36d 45e 78b 79e 7ac 89c 9ab bcd

\-\ 012 034 056 135 146 27c 28a 36d 45e 78b 79e 89c 9ab ade bcd

\-\ 012 034 056 135 146 236 24e 5de 78b 79e 7ac 89c 8ad 9ab bcd

\-\ 012 034 056 135 146 24d 25e 36e 78b 79e 7ac 89c 8ad 9ab bcd

\-\ 012 034 056 135 146 236 28c 45e 78b 79e 7ac 89d 9ab ade bcd

\-\ 012 034 056 135 146 25c 2be 368 48c 789 7ab 7de 9ae 9bd acd

\-\ 012 034 056 135 147 23b 245 68d 6ae 79e 7ac 89a 8bc 9bd cde

\-\ 012 034 056 135 146 27d 2be 369 45c 789 7ab 8bc 8de 9ae acd
\end{config}
\section*{Acknowledgements}
The third author acknowledges support from the APVV Research Grants 15-0220 and 17-0428, and the VEGA Research Grants 1/0142/17 and 1/0238/19.

We thank the anonymous referees for pointing out to us the work of Funk et alia~\cite{funk2003det} which completed the existence spectrum of weakly 3-chromatic configurations. 


\end{document}